\documentclass[12pt,reqno]{amsart}
\usepackage{amsmath,amsthm,amssymb,amscd,url}

\begin{document}

\allowdisplaybreaks

\long\def\FOOTNOTE#1{ \ignorespaces}

\newcommand{\TITLE}{Amicable pairs and aliquot cycles for elliptic curves}
\newcommand{\TITLERUNNING}{Aliquot cycles for elliptic curves}
\newcommand{\DATE}{\today}
\newcommand{\VERSION}{5}

\theoremstyle{plain} 
\newtheorem{theorem}{Theorem} 
\newtheorem{conjecture}[theorem]{Conjecture}
\newtheorem{proposition}[theorem]{Proposition}
\newtheorem{lemma}[theorem]{Lemma}
\newtheorem{corollary}[theorem]{Corollary}

\theoremstyle{definition}
\newtheorem*{definition}{Definition}

\theoremstyle{remark}
\newtheorem{remark}[theorem]{Remark}
\newtheorem{example}[theorem]{Example}
\newtheorem{question}[theorem]{Question}
\newtheorem*{acknowledgement}{Acknowledgements}

\def\BigStrut{\vphantom{$(^{(^(}_{(}$}} 

\newenvironment{notation}[0]{%
  \begin{list}%
    {}%
    {\setlength{\itemindent}{0pt}
     \setlength{\labelwidth}{4\parindent}
     \setlength{\labelsep}{\parindent}
     \setlength{\leftmargin}{5\parindent}
     \setlength{\itemsep}{0pt}
     }%
   }%
  {\end{list}}

\newenvironment{parts}[0]{%
  \begin{list}{}%
    {\setlength{\itemindent}{0pt}
     \setlength{\labelwidth}{1.5\parindent}
     \setlength{\labelsep}{.5\parindent}
     \setlength{\leftmargin}{2\parindent}
     \setlength{\itemsep}{0pt}
     }%
   }%
  {\end{list}}
\newcommand{\Part}[1]{\item[\upshape#1]}

%
\newcommand{\EndProofAtDisplay}{\renewcommand{\qedsymbol}{}}
\newcommand{\qedtag}{\tag*{\qedsymbol}}

\renewcommand{\a}{\alpha}
\renewcommand{\b}{\beta}
\newcommand{\g}{\gamma}
\renewcommand{\d}{\delta}
\newcommand{\e}{\epsilon}
\newcommand{\f}{\phi}
\newcommand{\fhat}{{\hat\phi}}
\renewcommand{\l}{\lambda}
\renewcommand{\k}{\kappa}
\newcommand{\lhat}{\hat\lambda}
\newcommand{\m}{\mu}
\renewcommand{\o}{\omega}
\renewcommand{\r}{\rho}
\newcommand{\rbar}{{\bar\rho}}
\newcommand{\s}{\sigma}
\newcommand{\sbar}{{\bar\sigma}}
\renewcommand{\t}{\tau}
\newcommand{\z}{\zeta}

\newcommand{\D}{\Delta}
\newcommand{\F}{\Phi}
\newcommand{\G}{\Gamma}

\newcommand{\ga}{{\mathfrak{a}}}
\newcommand{\gb}{{\mathfrak{b}}}
\newcommand{\gc}{{\mathfrak{c}}}
\newcommand{\gd}{{\mathfrak{d}}}
\newcommand{\gk}{{\mathfrak{k}}}
\newcommand{\gm}{{\mathfrak{m}}}
\newcommand{\gn}{{\mathfrak{n}}}
\newcommand{\gp}{{\mathfrak{p}}}
\newcommand{\gq}{{\mathfrak{q}}}
\newcommand{\gI}{{\mathfrak{I}}}
\newcommand{\gJ}{{\mathfrak{J}}}
\newcommand{\gK}{{\mathfrak{K}}}
\newcommand{\gM}{{\mathfrak{M}}}
\newcommand{\gN}{{\mathfrak{N}}}
\newcommand{\gP}{{\mathfrak{P}}}
\newcommand{\gQ}{{\mathfrak{Q}}}

\def\Acal{{\mathcal A}}
\def\Bcal{{\mathcal B}}
\def\Ccal{{\mathcal C}}
\def\Dcal{{\mathcal D}}
\def\Ecal{{\mathcal E}}
\def\Fcal{{\mathcal F}}
\def\Gcal{{\mathcal G}}
\def\Hcal{{\mathcal H}}
\def\Ical{{\mathcal I}}
\def\Jcal{{\mathcal J}}
\def\Kcal{{\mathcal K}}
\def\Lcal{{\mathcal L}}
\def\Mcal{{\mathcal M}}
\def\Ncal{{\mathcal N}}
\def\Ocal{{\mathcal O}}
\def\Pcal{{\mathcal P}}
\def\Qcal{{\mathcal Q}}
\def\Rcal{{\mathcal R}}
\def\Scal{{\mathcal S}}
\def\Tcal{{\mathcal T}}
\def\Ucal{{\mathcal U}}
\def\Vcal{{\mathcal V}}
\def\Wcal{{\mathcal W}}
\def\Xcal{{\mathcal X}}
\def\Ycal{{\mathcal Y}}
\def\Zcal{{\mathcal Z}}

\renewcommand{\AA}{\mathbb{A}}
\newcommand{\BB}{\mathbb{B}}
\newcommand{\CC}{\mathbb{C}}
\newcommand{\FF}{\mathbb{F}}
\newcommand{\GG}{\mathbb{G}}
\newcommand{\NN}{\mathbb{N}}
\newcommand{\PP}{\mathbb{P}}
\newcommand{\QQ}{\mathbb{Q}}
\newcommand{\RR}{\mathbb{R}}
\newcommand{\ZZ}{\mathbb{Z}}

\def \bfa{{\mathbf a}}
\def \bfb{{\mathbf b}}
\def \bfc{{\mathbf c}}
\def \bfe{{\mathbf e}}
\def \bff{{\mathbf f}}
\def \bfF{{\mathbf F}}
\def \bfg{{\mathbf g}}
\def \bfn{{\mathbf n}}
\def \bfp{{\mathbf p}}
\def \bfr{{\mathbf r}}
\def \bfs{{\mathbf s}}
\def \bft{{\mathbf t}}
\def \bfu{{\mathbf u}}
\def \bfv{{\mathbf v}}
\def \bfw{{\mathbf w}}
\def \bfx{{\mathbf x}}
\def \bfy{{\mathbf y}}
\def \bfz{{\mathbf z}}
\def \bfX{{\mathbf X}}
\def \bfU{{\mathbf U}}
\def \bfmu{{\boldsymbol\mu}}

\newcommand{\Gbar}{{\bar G}}
\newcommand{\Kbar}{{\bar K}}
\newcommand{\kbar}{{\bar k}}
\newcommand{\Obar}{{\bar O}}
\newcommand{\Pbar}{{\bar P}}
\newcommand{\Rbar}{{\bar R}}
\newcommand{\Qbar}{{\bar{\QQ}}}


\newcommand{\Aut}{\operatorname{Aut}}
\newcommand{\Ctilde}{{\tilde C}}
\newcommand{\defeq}{\stackrel{\textup{def}}{=}}
\newcommand{\Disc}{\operatorname{Disc}}
\renewcommand{\div}{\operatorname{div}}
\newcommand{\Div}{\operatorname{Div}}
\newcommand{\et}{{\text{\'et}}}  
\newcommand{\Etilde}{{\tilde E}}
\newcommand{\End}{\operatorname{End}}
\newcommand{\EDS}{\mathsf{W}}
\newcommand{\Fix}{\operatorname{Fix}}
\newcommand{\Frob}{\operatorname{Frob}}
\newcommand{\Gal}{\operatorname{Gal}}
\newcommand{\GCD}{{\operatorname{GCD}}}
\renewcommand{\gcd}{{\operatorname{gcd}}}
\newcommand{\GL}{\operatorname{GL}}
\newcommand{\hhat}{{\hat h}}
\newcommand{\Hom}{\operatorname{Hom}}
\newcommand{\Ideal}{\operatorname{Ideal}}
\newcommand{\Image}{\operatorname{Image}}
\newcommand{\Jac}{\operatorname{Jac}}
\newcommand{\longhookrightarrow}{\lhook\joinrel\relbar\joinrel\rightarrow}
\newcommand{\MOD}[1]{~(\textup{mod}~#1)}
\renewcommand{\pmod}{\MOD}
\newcommand{\Norm}{\operatorname{N}}
\newcommand{\NS}{\operatorname{NS}}
\newcommand{\notdivide}{\nmid}
\newcommand{\OK}[1]{\Ocal_{K,#1}^{\,\sharp}}
\newcommand{\longonto}{\longrightarrow\hspace{-10pt}\rightarrow}
\newcommand{\onto}{\rightarrow\hspace{-10pt}\rightarrow}
\newcommand{\ord}{\operatorname{ord}}
\newcommand{\Parity}{\operatorname{Parity}}
\newcommand{\pequiv}{\stackrel{pr}{\equiv}}
\newcommand{\Pic}{\operatorname{Pic}}
\newcommand{\Prob}{\operatorname{Prob}}
\newcommand{\Proj}{\operatorname{Proj}}
\newcommand{\rank}{\operatorname{rank}}
\newcommand{\res}{\operatornamewithlimits{res}}
\newcommand{\Resultant}{\operatorname{Resultant}}
\renewcommand{\setminus}{\smallsetminus}
\newcommand{\stilde}{{\tilde\sigma}}
\newcommand{\sign}{\operatorname{Sign}}
\newcommand{\Spec}{\operatorname{Spec}}
\newcommand{\Support}{\operatorname{Support}}
\newcommand{\tors}{{\textup{tors}}}
\newcommand{\Tr}{\operatorname{Tr}}
\newcommand{\Trace}{\operatorname{Tr}}
\newcommand\W{W^{\vphantom{1}}}
\newcommand{\Wtilde}{{\widetilde W}}
\newcommand{\<}{\langle}
\renewcommand{\>}{\rangle}
\newcommand{\semiquad}{\hspace{.5em}}

\newcommand{\CS}[2]{\genfrac(){}{}{#1}{#2}_3}  
\newcommand{\FCS}[2]{\genfrac[]{}{}{#1}{#2}_3}  
\newcommand{\LS}[2]{\genfrac(){}{}{#1}{#2}}  
\newcommand{\FLS}[2]{\genfrac[]{}{}{#1}{#2}}  
\renewcommand{\SS}[2]{\genfrac(){}{}{#1}{#2}_6}  

\let\hw\hidewidth

\hyphenation{para-me-tri-za-tion}


\title[\TITLERUNNING]{\TITLE}
\date{\DATE, Draft \#\VERSION}
\author{Joseph H. Silverman}
\address{Mathematics Department, Box 1917, Brown University, 
Providence, RI 02912 USA} 
\email{jhs@math.brown.edu} 
\author{Katherine E. Stange} 
\address{%
Department of Mathematics, Simon Fraser University,
8888 University Drive, Burnaby, BC, Canada V5A 1S6,
and 
Pacific Institute for the Mathematical Sciences,
200 1933 West Mall, Vancouver, BC, Canada V6T 1Z2}
\email{kestange@pims.math.ca}

\thanks{The first author's research supported by 
NSF DMS-0650017 and DMS-0854755.
The second author's research supported by NSERC PDF-373333.
}


\begin{abstract}
An \emph{amicable pair} for an elliptic curve~$E/\QQ$ is a pair of
primes~$(p,q)$ of good reduction for~$E$ satisfying
$\#\Etilde_p(\FF_p) = q$ and $\#\Etilde_q(\FF_q) = p$.  In this paper
we study elliptic amicable pairs and analogously defined longer
\emph{elliptic aliquot cycles}. We show that there exist elliptic
curves with arbitrarily long aliqout cycles, but that~CM elliptic
curves (with $j\ne0)$ have no aliqout cycles of length greater than
two.  We give conjectural formulas for the frequency of amicable
pairs.  For CM curves, the derivation of precise conjectural formulas
involves a detailed analysis of the values of the Gr\"ossencharacter
evaluated at primes~$\gp$ in~$\End(E)$ having the property
that~$\#\Etilde_\gp(\FF_\gp)$ is prime. This is especially intricate
for the family of curves with $j=0$.
\end{abstract}

\maketitle



\section*{Introduction}
\label{section:introduction}

Let $E/\QQ$ be an elliptic curve. In this paper we study pairs of
primes~$(p,q)$ such that~$E$ has good reduction at~$p$ and~$q$ and
such that the reductions~$\Etilde_p$ and~$\Etilde_q$ of~$E$ at~$p$
and~$q$ satisfy
\[
  \#\Etilde_p(\FF_p) = q
  \qquad\text{and}\qquad
  \#\Etilde_q(\FF_q) = p.
\]
By analogy with a classical problem in number theory
(cf.\ Remark~\ref{remark:classical}), we call~$(p,q)$ an
\emph{amicable pair} for the elliptic curve~$E/\QQ$.

\begin{example}
\label{example:cond3743}
Searching for amicable pairs using primes
smaller than~$10^7$ on the two elliptic curves
\[
  E_1:y^2+y=x^3-x\quad\text{and}\quad
  E_2:y^2+y=x^3+x^2,
\]
yields one amicable pair on the curve $E_1$,
\[
  (1622311, 1622471),
\]
and four amicable pairs on the curve $E_2$,
\[
  (853, 883),\semiquad
  (77761, 77999),\semiquad
  (1147339, 1148359),\semiquad
  (1447429, 1447561).
\]
\end{example}

\begin{example}
\label{exmaple:CM32}
The curve
\[
  E_3:y^2=x^3+2
\]
exhibits strikingly different amicable pair behavior. There are more
than~$800$ amicable pairs for~$E_3$ using primes smaller that~$10^6$,
the first few of which are
\[
  (13, 19), (139, 163), (541, 571), (613, 661), (757, 787), (1693, 1741).
\]
\end{example}

One objective of this note is to present theoretical and
numerical evidence for the following conjecture.

\begin{conjecture}
\label{conj:amicablepair}
Let $E/\QQ$ be an elliptic curve, let
\begin{align*}
  \Qcal_E(X)
   &= \#\bigl\{\text{amicable pairs $(p,q)$ for $E/\QQ$ with $p<q$ and
          $p\le X$}\bigr\}.
\end{align*}
be the amicable pair counting function, and assume that there are
infinitely many primes~$p$ such that~$\#\Etilde_p(\FF_p)$ is prime.
\begin{parts}
\Part{(a)}
If $E$ does not have complex multiplication, then
\[
  \Qcal_E(X) \gg\ll \frac{\sqrt{X}}{(\log X)^2}
  \quad\text{as $X\to\infty$,}
\]
where the implied constants depend on~$E$.
\Part{(b)}
If $E$ has complex multiplication, then there is a constant~$A_E>0$
such that
\[
  \Qcal_E(X)\sim A_E\frac{X}{(\log X)^2}.
\]
\end{parts}
\end{conjecture}

We do not believe that it is clear, \emph{a priori}, why there should
be such a striking difference between the CM and the non-CM cases. We
first discovered this phenomenon experimentally; subsequently we found
an explanation based on Theorem~\ref{thm:amicableCMconj}, which says
that if~$E/\QQ$ has~CM and if~$q=\#\Etilde_p(\FF_p)$ is prime, then
there are generally only two possible values for~$\#\Etilde_q(\FF_q)$,
one of which is~$p$. (The situation for $j(E)=0$ is considerably more
complicated; see Section~\ref{section:amicableCMj0}.)  This contrasts
with the non-CM case, where~$\#\Etilde_q(\FF_q)$ seems to be free to
range throughout the Hasse interval. We refer the reader to
Conjectures~\ref{conj:CMamicable} and~\ref{conjecture:N1density} for more
precise versions of the CM part of Conjecture~\ref{conj:amicablepair}.

The frequency of primes~$p$ such that $\#\Etilde_p(\FF_p)$ is prime or
almost prime has been studied by a number of authors. In
Section~\ref{section:Epprimeinfoften} we discuss what is known and
what is conjectured concerning this problem.

Generalizing the notion of amicable pair, we define an \emph{aliquot
cycle} of length~$\ell$ for $E/\QQ$ to be a sequence of distinct
primes $(p_1,p_2,\ldots,p_\ell)$ such that~$E$ has good reduction at
every~$p_i$ and such that
\begin{multline*}
  \#\Etilde_{p_1}(\FF_{p_1})=p_2,\quad   \#\Etilde_{p_2}(\FF_{p_2})=p_3,
      \quad\ldots\quad \\*
  \#\Etilde_{p_{\ell-1}}(\FF_{p_{\ell-1}})=p_\ell,\quad
  \#\Etilde_{p_\ell}(\FF_{p_\ell})=p_1.
\end{multline*}

\begin{example}
\label{example:aliquotcycles}
The elliptic curve $y^2=x^3-25x-8$ has the aliquot triple
$(83, 79, 73)$. The elliptic curve
{\small\[
  E:y^2=x^3+176209333661915432764478x+ 60625229794681596832262
\]}%
has an aliquot cycle $(23, 31, 41, 47, 59, 67, 73, 79, 71, 61, 53, 43,
37, 29)$ of length~$14$.
\end{example}

In Section~\ref{section:nonCM} we give an heuristic argument
suggesting that the counting function for aliquot cycles of
length~$\ell$ for non-CM elliptic curves grows like $\sqrt{X}/(\log
X)^\ell$.  The rough idea is to assume that if $q=\#\Etilde_p(\FF_p)$
is prime, then the trace values $a_q(E)=q+1-\#\Etilde_q(\FF_q)$ are
(more-or-less) equidistributed within the appropriate Hasse interval.
\par
In Section~\ref{section:longaliquotcycles} we give an elementary
construction (Theorem~\ref{theorem:arblngth}) using the prime number
theorem, the Chinese remainder theorem, and a result of Deuring,
to prove that for every~$\ell$ there exists an elliptic curve~$E/\QQ$
with an aliquot cycle of length~$\ell$.
\par
We next consider the case of elliptic curves having complex
multiplication. These curves exhibit strikingly different behavior
from their non-CM counterparts.
Our first result (Theorem~\ref{thm:amicableCMconj}) 
says that if~$E/\QQ$ has~CM with
$j(E)\ne0$, and if $q=\#\Etilde_p(\FF_p)$ is prime, then there are
only two possible values for $\#\Etilde_q(\FF_q)$, namely~$p$
and~$2q+2-p$. Assuming each is equally likely (which seems to be the
case experimentally), this explains why~CM curves have so many
amicable pairs. The proof involves first proving
that~$p$ and~$q$ split in~$\End(E)$, and then relating the values of
the Gr\"ossencharacter~$\psi_E$ at primes lying above~$p$ and~$q$.
\par
Theorem~\ref{thm:amicableCMconj} can also be used to show that a CM curve with
$j\ne0$ has no aliquot cycles of length~$3$ or greater; see
Corollary~\ref{corollary:CMhasnoaliqge3}.  This stands in contrast to
Theorem~\ref{theorem:arblngth}, which says that there exist curves
with arbitrarily long aliquot cycles.
\par
We finally turn to the $j=0$ curves $y^2=x^3+k$,
whose complicated analysis is given in a
lengthy Section~\ref{section:amicableCMj0}.
For prime values of~$k$, we give a precise conjectural formula
for the counting function of amicable pairs that depends on 
the value of~$k$ modulo~$36$. For example, if~$k$ is prime
and $k\equiv1~\text{or}~19\pmod{36}$, then we conjecture that
\begin{equation}
  \label{eqn:introlimit}
  \lim_{X\to\infty}
  \frac{ \#\bigl\{p<X : \text{$p$ is part of an amicable pair} \bigr\} }
       { \#\bigl\{p<X : \text{$\#\Etilde_p(\FF_p)$ is prime} \bigr\} }
  = \frac16 + \frac{1}{3k-9},
\end{equation}
while if $k\equiv11~\text{or}~23\pmod{36}$, then the
limiting value in~\eqref{eqn:introlimit}
is (conjecturally) equal to $\frac16 + \frac{k}{3k^2-6}$. There are similar
formulas for the other congruence classes.
\par
The derivation of these formulas is in two parts. First, by
analyzing the values of the Gr\"ossencharacter and using sextic reciprocity,
we prove that~$(p,q)$ is an amicable pair if and only if
$\SS{\psi_E(\gp)}{k}\SS{1-\psi_E(\gp)}{k}=1$. If the values
of~$\psi_E(\gp)$ modulo~$k$ were equidistributed as~$p$ varies, we
would conjecture that the number of amicable pairs is governed
by the proportion of $\l\in\Ocal/k\Ocal$ satisfying $\SS{\l(1-\l)}{k}=1$.
(Here $\Ocal=\End(E)=\ZZ[(1+\sqrt{-3})/2]$.) This is almost true, but
the allowable values of~$\l$ are often restricted by 
further conditions on~$\SS{\l}{k}$. Sorting out these restrictions
gives a precise conjectural value for the limit~\eqref{eqn:introlimit}
in terms of the sizes of certain subsets of~$\Ocal/k\Ocal$. 
\par
The second part of the proof is to derive explicit formulas for the
sizes of these sets. This is done by relating the points in
these sets to the
$\Ocal/k\Ocal$-points on a certain family of
curves~$C^{(\g,\d)}$ of genus four. We count these points by explicitly
decomposing the Jacobian of~$C^{(\g,\d)}$ into a product of four $j=0$
elliptic curves and using the Gr\"ossencharacter formula for the
number of points on such curves. The resulting formulas are quite
involved, especially in the case that~$k$ splits in~$\Ocal$, but
eventually most of the terms cancel, leaving a relatively compact
formula. We have no good explanation for why the final formula
has such a simple form; see Remark~\ref{remark:Mksplitcomputation} 
for a discussion of the delicacy of the computation.
\par
The conjectures in this paper are supported by heuristic arguments
and, especially for~CM curves, by theorems describing the allowable
values of the Gr\"ossencharacter~$\psi_E$. But heuristic arguments
have been known to fail, and indeed our~CM argument depends on the
assumption that~$\psi_E(\gp)\bmod k$ is uniformly distributed among
its \emph{allowable values}, where we claim to have characterized the
set of allowable values. It is thus reassuring that extensive
experiments are in close agreement with the conjectural values derived
by theory. These experiments are described in
Section~\ref{section:amicableCMexper}.
\par
Finally, in Section~\ref{section:tangentialremarks} we explain where
we first ran across amicable pairs and aliquot cycles for elliptic
curves, and we describe some possible generalizations that 
deserve further study.

\section{How often is $\#\Etilde_p(\FF_p)$ prime?}
\label{section:Epprimeinfoften}
If an elliptic curve~$E/\QQ$ is to have any amicable pairs or aliquot
cycles, then it is clearly necessary that there exist primes~$p$
such that $\Etilde_p(\FF_p)$ is prime. The question of existence and
density of such primes has been studied by various authors.  

\begin{remark}
\label{remark:noaliqiftors}
If $E(\QQ)_\tors\ne\{O\}$, then $\#\Etilde_p(\FF_p)$ will
be composite for all but finitely many~$p$, since
$E(\QQ)_\tors\hookrightarrow \Etilde_p(\FF_p)$ for all
$p\notdivide2\D_{E/\QQ}$.  Using this observation, it is
quite easy to produce curves having no nontrivial aliquot cycles, for
example, the curves $y^2=x^3+x$ and $y^2=x^3+1$. 
\end{remark}

More generally, there may be a local obstruction associated with the
representation $\Gal(\Qbar/\QQ)\to\Aut(E_\tors)$ that
forces~$\#\Etilde_p(\FF_p)$ to be composite for all but finitely
many~$p$; see~\cite{Zywina}. We quote a special case of a conjecture
of Koblitz, as modified by Zywina.

\begin{conjecture}
\label{conj:kobzy}
\textup{(Koblitz~\cite{Koblitz}, Zywina~\cite{Zywina})}
Let $E/\QQ$ be an  elliptic curve, and let
\[
   \Ncal_E(X) = \#\bigl\{\text{primes $p\le X$ such that $\#\Etilde_p(\FF_p)$
       is prime}\bigr\}
\]
count how often~$E$ modulo~$p$ has a prime number of points.  Then
there is a constant $C_{E/\QQ}$ such that
\[
  \Ncal_E(X) \sim C_{E/\QQ}\frac{X}{(\log X)^2}.
\]
Further, $C_{E/\QQ}>0$ if and only if there are infinitely many
primes~$p$ such that $\#\Etilde_p(\FF_p)$ is prime.
\end{conjecture}

Koblitz and Zywina give formulas for the constant~$C_{E/\QQ}$ in terms
of the image of the representation $\Gal(\Qbar/\QQ)\to\Aut(E_\tors)$.
In principle this allows one to approximate~$C_{E/\QQ}$ to high
precision, and they give a number of examples. For additional material
on the probability that~$\#\Etilde_p(\FF_p)$ is prime or almost prime,
see~\cite{Cojocaru} and~\cite{Jimenez}.

\section{Aliquot cycles and amicable pairs for elliptic curves}
\label{section:defs}

We formally give the following definitions as previously described in
the introduction.

\begin{definition}
Let $E/\QQ$ be an elliptic curve.  An \emph{aliquot cycle of
length~$\ell$} for~$E/\QQ$ is a sequence of distinct primes
$(p_1,p_2,\ldots,p_\ell)$ such that~$E$ has good reduction at
every~$p_i$ and such that
\begin{multline*}
  \#\Etilde_{p_1}(\FF_{p_1})=p_2,\quad   \#\Etilde_{p_2}(\FF_{p_2})=p_3,
      \quad\ldots\quad \\*
  \#\Etilde_{p_{\ell-1}}(\FF_{p_{\ell-1}})=p_\ell,\quad
  \#\Etilde_{p_\ell}(\FF_{p_\ell})=p_1.
\end{multline*}
An aliquot cycle is \emph{normalized} if $p_1=\min p_i$. Every aliquot
cycle can be normalized by a cyclic shift of its elements.
An \emph{amicable pair} is an aliquot cycle of length two.
\end{definition}

\begin{remark}
\label{remark:classical}
Classically, an amicable pair is a pair of integers $(m,n)$ satisfying
$\stilde(m)=n$ and $\stilde(n)=m$, where~$\stilde(n)=\s(n)-n$ is the sum of
the proper divisors of~$n$. Similarly, a number~$n $ is perfect if
$\stilde(n)=n$, and a (classical) aliquot cycle is a list of
distinct integers $(n_1,n_2,\ldots,n_\ell)$ satisfying
\[
  \stilde(n_1)=n_2,\quad  \stilde(n_2)=n_3,\quad  \ldots\quad
  \stilde(n_{\ell-1})=n_\ell,\quad  \stilde(n_\ell)=n_1.
\]
(Numbers appearing in an aliquot cycle are also called sociable
numbers.)  Perfect numbers and amicable pairs were studied in ancient
Greece, and aliquot cycles of all lengths continue to attract interest
to the present day. See for
example~\cite{AmicableSurvey,teRiele1,teRiele2,Yan}.

By analogy with the classical case, one might call an aliquot
cycle~$(p)$ of length one for $E/\QQ$ a \emph{perfect prime}, but such
primes have already been given a name. They are called \emph{anomalous
  primes} and appear as exceptional cases in diverse applications; see
for example~\cite{Mazur,Olson}. In particular, anomalous primes 
are to be avoided in cryptography because the elliptic curve discrete
logarithm problem (ECDLP) for anomalous primes can be solved in linear
time~\cite{SatohAraki,Semaev,Smart}.
\end{remark}

We begin our study of aliquot cycles with the following general
observation concerning amicable pairs.

\begin{proposition}
\label{proposition:sameendring}
Let $E/\QQ$ be an elliptic curve, and let~$(p,q)$ be a normalized amicable pair
for~$E/\QQ$ with $p\ge5$. Then
\[
  \End(\Etilde_p/\FF_p)\otimes\QQ \cong
  \End(\Etilde_q/\FF_q)\otimes\QQ .
\]
\end{proposition}
\begin{proof}
The fact that $p$ is odd and $q=\#\Etilde_p(\FF_p)=p+1-a_p$ is prime
implies in particular that~$a_p\ne0$, so~$E$ has ordinary reduction
at~$p$.  (This is where we use the assumption that~$p\ge5$;
cf.\ \cite[Exercise~5.10]{AEC}.)  Reversing the roles of~$p$ and~$q$
shows that~$E$ also has ordinary reduction at~$q$.
\par 
The assumption that~$(p,q)$ is an amicable pair is equivalent to the
assertions that
\[
  q = p+1-a_p\qquad\text{and}\qquad p=q+1-a_q,
\]
and then a little bit of algebra shows that 
\begin{equation}
  \label{eqn:ap24peqaq24q}
  a_p^2 - 4p = a_q^2 - 4q.
\end{equation}
\par
The field~$\End(\Etilde_p/\FF_p)\otimes\QQ$ is generated by
the Frobenius element $\Frob_p(x)=x^p$, which is a root of
\[
  T^2 - a_p T + p = 0.
\]
Thus 
\[
  \End(\Etilde_p/\FF_p)\otimes\QQ \cong \QQ\left(\sqrt{a_p^2-4p}\;\right).
\]
The analogous formula is true for~$q$, and
then~\eqref{eqn:ap24peqaq24q} completes the proof of the proposition.
\end{proof}

\section{Counting aliquot cycles for non-CM elliptic curves}
\label{section:nonCM}

In this section we study the aliquot cycle counting function
\[
  \Qcal_{E,\ell}(X) = \#\left\{
     \begin{tabular}{@{}c@{}}
       normalized aliquot cycles $(p_1,\ldots,p_\ell)$ \\
       of length $\ell$ for $E/\QQ$ satisfying $p_1\le X$ \\
     \end{tabular}
      \right\}.
\]

\begin{conjecture}
\label{conjecture:aliquotcount}
Let~$E/\QQ$ be an elliptic curve that does not have complex
multiplication, and assume that there are infinitely many primes~$p$
such that~$\#\Etilde_p(\FF_p)$ is prime.  Then the aliquot cycle
counting function satisfies
\[
  \Qcal_{E,\ell}(X) \gg\ll \frac{\sqrt{X}}{(\log X)^\ell}
  \quad\text{as $X\to\infty$,}
\]
where the implied positive constants depend on~$E$ and~$\ell$, but are
independent of~$X$.
\end{conjecture}

\begin{remark} 
As noted in Section~\ref{section:tangentialremarks}, an aliquot
cycle~$(p)$ of length one consists of a single anomalous prime. In
this case, Conjecture~\ref{conjecture:aliquotcount} follows from a
general conjecture of Lang and Trotter~\cite{LangTrotter}, which predicts
the stronger result $\Qcal_{E,1}(X)\sim c\sqrt{X}/\log X$.
\end{remark}

We give an heuristic argument in support of
Conjecture~\ref{conjecture:aliquotcount}. To ease notation, let
\[
  N_p = \#\Etilde_p(\FF_p).
\]
Then, setting $p_1=p$, we have
\newcommand{\deq}{\stackrel{\text{def}}{=}}
\begin{align}
  \label{eqn:Pppoaac}
  \Prob&(\text{$p$ is part of an aliquot cycle of length $\ell$}) \notag\\*
  &= \Prob\left(\begin{tabular}{@{}c@{}}
     $p_2 \deq N_{p_1}$ is prime
          and $p_3 \deq N_{p_2}$ is prime and \\
      \dots\ and $p_\ell \deq N_{p_{\ell-1}}$ is prime and $N_{p_\ell}=p_1$\\
  \end{tabular}
  \right) \notag\\
  &\approx
    \biggl(\prod_{i=1}^{\ell-1} \Prob(\text{$p_{i+1} \deq N_{p_i}$ is prime})\biggr)
      \Prob(N_{p_\ell}=p_1).
\end{align}
(We ignore the small probabilty that there is some $i<\ell$ such
that~$N_{p_i}$ is equal to an earlier~$p_j$.)
\par
Under our assumption that~$N_p$ is prime for infinitely many~$p$,
Conjecture~\ref{conj:kobzy} says that
\[
  \Prob(\text{$N_p$ is prime}) \gg\ll \frac{1}{\log p},
\]
and since
\[
  p_{i+1}=N_{p_i} = p_i + O(\sqrt{p_i}),
\]
every term in the sequence $p=p_1,p_2,\ldots,p_\ell$ satisfies $p_i =
p + O(\sqrt p)$. Hence
\[
  \Prob(\text{$N_{p_i}$ is prime})
  \gg\ll \frac{1}{\log p_i}
  \sim \frac{1}{\log p}.
\]
Substituting this into~\eqref{eqn:Pppoaac} gives
\begin{equation}
  \label{eqn:Pppaliq}
  \Prob\left(\begin{tabular}{@{}c@{}}
     $p$ is part of an aliquot\\cycle of length $\ell$\\
   \end{tabular} \right)
  \approx \frac{1}{(\log p)^{\ell-1}} \cdot \Prob(N_{p_\ell}=p_1).
\end{equation}

In order to estimate the last factor, we use the Sato--Tate
conjecture~\cite[C.21.1]{AEC}, which says that as~$q$ varies, the values
of~$N_q$ are distributed in the
interval~$[q+1-2\sqrt{q},q+1+2\sqrt{q}]$ according to the Sato--Tate
distribution,
\[
  \#\left\{q \le X : a \le \frac{q+1-N_q}{2\sqrt{q}} \le b\right\}
  \sim \pi(X)\cdot\frac{2}{\pi}\int_a^b\sqrt{1-t^2}\,dt.
\]
(See~\cite{Taylor} for a proof of the Sato--Tate conjecture in certain
cases, although our use of the conjecture is purely heuristic.)
Then for primes~$p$ and $q=p+O(\sqrt{p})$, a rough estimate gives
\begin{equation}
  \label{eqn:PNqp}
  \Prob(N_q=p) \gg\ll \frac{1}{\sqrt{q}} \sim \frac{1}{\sqrt{p}}.
\end{equation}
Combining~\eqref{eqn:Pppaliq} and~\eqref{eqn:PNqp} yields
\[
  \Prob\left(\begin{tabular}{@{}c@{}}
     $p$ is part of an aliquot\\cycle of length $\ell$\\
   \end{tabular} \right)
  \gg\ll \frac{1}{\sqrt{p}(\log p)^{\ell-1}} .
\]

We now estimate the number of normalized aliquot cycles of length~$\ell$ 
whose initial prime is less than~$X$.  
\begin{align*}
  \Qcal_{E,\ell}(X)
  &\approx \sum_{p\le X} 
      \Prob\left(\begin{tabular}{@{}c@{}}
         $p$ is the initial element of a normalized\\
         aliquot cycle of length $\ell$\\
       \end{tabular} \right) \\
  &\gg\ll \sum_{p\le X} \frac{1}{\sqrt{p}(\log p)^{\ell-1}} .
\end{align*}
It only remains to use the rough approximation
\[
  \sum_{p\le X} f(X) \approx \sum_{n\le X/\log X} f(n\log n)
  \approx \int^{X/\log X} \hspace{-1.5em} f(t\log t)\,dt
  \approx \int^X \hspace{-.5em} f(u)\,\frac{du}{\log u}
\]
to obtain
\[
  \Qcal_{E,\ell}(X)
  \gg\ll 
  \int^X \frac{1}{\sqrt{u}(\log u)^{\ell-1}} \cdot \frac{du}{\log u}
  \gg\ll
  \frac{\sqrt X}{(\log X)^\ell}.
\]

\section{Aliquot cycles of arbitrary length}
\label{section:longaliquotcycles}

\begin{theorem}
\label{theorem:arblngth}
For every $\ell\ge1$ there exists an elliptic curve~$E/\QQ$ that has
an aliquot cycle of length~$\ell$. More generally, for any positive
integers~$\ell_1,\ldots,\ell_r$ there exists an elliptic curve~$E/\QQ$
that has distinct aliquot cycles of lengths~$\ell_1,\ldots,\ell_r$.
\end{theorem}
\begin{proof}
A theorem of Deuring~\cite{Deuring} (vastly
generalized by Waterhouse~\cite{Waterhouse}, see also
R\"uck~\cite{Ruck}) says that if~$p$ is a prime and~$t$ is an integer
satisfying~$|t|\le2\sqrt{p}$, then there exists an elliptic curve
$\Etilde/\FF_p$ satisfying
\[
  \#\Etilde(\FF_p) = p+1-t.
\]
In other words, every Frobenius trace in the Hasse interval for~$p$
actually occurs as the trace of an elliptic curve defined
over~$\FF_p$.
\par
Now fix~$\ell$ and let $p_1,p_2,\ldots,p_\ell$ be a sequence of
primes with the property that
\begin{equation}
  \label{eqn:pipi1le2pi}
  |p_i+1-p_{i+1}|\le 2\sqrt{p_i}
  \quad\text{for all $1\le i\le \ell$,}
\end{equation}
where by convention we set~$p_{\ell+1}=p_1$. It is easy enough to find
such a sequence. To be precise, we can use a weak form of the prime
number theorem~\cite[Theorem~4.7]{Apostol} that says that there are positive
constants~$a$ and~$b$ such that the $n^{\text{th}}$~prime~$q_n$
satisfies
\[
  a n\log(n) \le q_n \le b n\log(n).
\]
It follows that for any given~$\ell$, if we choose~$n$ to be
sufficiently large, then
\[
  q_{n+\ell} - q_n - 1 \le 2\sqrt{q_n}.
\]
This implies that the sequence of
primes~$(q_{n+1},q_{n+2},\ldots,q_{n+\ell})$
satisfies~\eqref{eqn:pipi1le2pi}, so we take this to be our
sequence~$(p_1,\ldots,p_\ell)$.
\par
Applying the theorem of Deuring cited earlier, for each~$p_i$ we can find
an elliptic curve~$\Etilde_i/\FF_{p_i}$ satisfying
\[
  \#\Etilde_i(\FF_{p_i}) = p_{i+1}.
\]
(This includes the case $i=\ell$, in which case $p_{\ell+1}=p_1$.) We now use
the Chinese remainder theorem on the coefficients of the Weierstrass equations
for~$\Etilde_1,\ldots,\Etilde_\ell$ to find an elliptic curve~$E/\QQ$ satisfying
\[
  E \bmod p_i \cong \Etilde_i\quad\text{for all $1\le i\le \ell$.}
\]
Then by construction, the sequence~$(p_1,\ldots,p_\ell)$ is an aliquot
cycle of length~$\ell$ for~$E/\QQ$.
\par
In a similar fashion, we can construct elliptic curves over~$\QQ$ that
have aliquot cycles of any specified lengths using different sets of
primes, and then we can Chinese remainder these curves to obtain a
single elliptic curve over~$\QQ$ with any specified number of aliquot
cycles of any specified lengths.
\end{proof}

\begin{remark}
The algorithm described in Theorem~\ref{theorem:arblngth} works well
in practice, although it naturally yields equations having large
coefficients. We used it in Example~\ref{example:aliquotcycles} to
find an aliquot cycle of length~$14$. Here's another example.  The
following elliptic curve has an aliquot cycle of length~$25$, starting
with the prime~$p=41$.
{\par\tiny
\begin{multline*}
  y^2 = x^3 + 4545482133607498579268567738514832922289740324532 x\\
    + 595867265462112118291430245894379464967885794713.
\end{multline*}
}%
\end{remark}

\section{Amicable pairs for CM curves with $j\ne0$}
\label{section:amicableCM}

Our next goal is to formulate and provide evidence for more precise
versions of the~CM part of Conjecture~\ref{conj:amicablepair}.  A key
observation is that if~$E$ has~CM, then the assumption that
$q=\#\Etilde_p(\FF_p)$ is prime severely limits the possible values of
$\Etilde_q(\FF_q)$.  It turns out that the case of elliptic curves
with $j(E)=0$ is significantly more complicated than the other cases,
so we deal with the $j(E)\ne0$ curves in this section and leave the
$j(E)=0$ curves for the next section.

\begin{theorem}
\label{thm:amicableCMconj}
Let $E/\QQ$ be an elliptic curve and assume\textup:
\begin{parts}
\Part{(1)}
$E$ has complex multiplication by an order $\Ocal$ 
in a quadratic imaginary field~$K=\QQ(\sqrt{-D})$.
\Part{(2)}
$p$ and $q$ are primes of good reduction for~$E$
with $p\ge5$ and
\[
  q=\#\Etilde_p(\FF_p).
\]
\Part{(3)}
$j(E)\ne0$, or equivalently, 
$\Ocal\ne\ZZ\left[\frac{1+\sqrt{-3}}{2}\right]$.
\end{parts}
Then~$D\equiv3\pmod4$, and either
\[
  \#\Etilde_q(\FF_q)=p
  \qquad\text{or}\qquad
  \#\Etilde_q(\FF_q)=2q+2-p.
\]
\end{theorem}

Theorem~\ref{thm:amicableCMconj} has an interesting consequence
concerning the allowable lengths of aliquot cycles for CM elliptic
curves. This may be compared with
Theorem~\ref{section:longaliquotcycles}, which says that there exist
(necessarily non-CM) curves having aliqout cycles of arbitrary length,
and with Conjecture~\ref{conjecture:aliquotcount}, which implies that
every non-CM elliptic curve has aliqout cycles of arbitrary length
provided that there are infinitely many primes~$p$ such that
$\#\Etilde_p(\FF_p)$ is prime.

\begin{corollary}
\label{corollary:CMhasnoaliqge3}
A CM elliptic curve $E/\QQ$ with $j(E)\ne0$ has no aliquot cycles of
length $\ell\ge3$ consisting of primes $p\ge5$.
\end{corollary}


\begin{remark}
There are various ways in which one might generalize
Theorem~\ref{thm:amicableCMconj}.  For example, 
replacing assumption~(2) by the assumption that~$L$ is an
integer such that the quantity
\[
  q = L^2 - \bigl(p+1-\#\Etilde_p(\FF_p)\bigr)L+p
\]
is prime and splits in~$K$ leads to the following conclusion:
\[
  a_q(E) = \pm \bigl(a_p(E)+2L\bigr).
\]
Theorem~\ref{thm:amicableCMconj} is the case $L=1$. We omit the proof
of the generalization, since it is similar and is not required in this
paper.
\end{remark}

\begin{remark}
\label{remark:aliquotcyclej0}
Corollary~\ref{corollary:CMhasnoaliqge3} omits curves with~$j(E)=0$.
It turns out that $j=0$ curves possess a rich and complicated amicable
pair structure which will be investigated in detail in
Section~\ref{section:amicableCMj0}.
Corollary~\ref{corollary:j0explicit} gives an analogue of
Theorem~\ref{thm:amicableCMconj} saying that there are (often) six
possible values for $\#\Etilde_q(\FF_q)$, rather than only the two
possibilities given in Theorem~\ref{thm:amicableCMconj}.  Using this
result, we are able to prove by a detailed case-by-case analysis that
$j=0$ curves cannot have aliquot cycles of length three; see
Appendix~\ref{appendix:j0aliqtriples}. But we do not have a proof that
there are no aliquot cycles of length greater than three.
\end{remark}

Before commencing the proofs of Theorem~\ref{thm:amicableCMconj} and
Corollary~\ref{corollary:CMhasnoaliqge3}, we prove a basic result
concerning the splitting of primes in CM fields.

\begin{lemma}
\label{lemma:psplits}
Let $E/\QQ$ be an elliptic curve with complex multiplication by~$K$,
let~$p\ge5$ be a prime of good reduction for~$E/\QQ$, and suppose that
$\#\Etilde_p(\FF_p)$ is odd. Then~$p$ splits in~$K$.
\end{lemma}
\begin{proof}
We have
\[
  \#\Etilde_p(\FF_p) = p + 1 - a_p,
\]
so the assumptions that~$p\ne2$ and~$\#\Etilde_p(\FF_p)$ is odd imply
that~$a_p$ is odd, so in particular $a_p\ne0$.  Hence~$E$ has ordinary
reduction at~$p$.  (Note that our assumption that $p\ge5$ and Hasse's
bound $|a_p|\le2\sqrt{p}$ imply that $p\mid a_p$ if and only if
$a_p=0$.)  It follows that the field~$K$ is isomorphic
to~\text{$\End(\Etilde_p)\otimes\QQ$}, which is generated by a root of
the characteristic polynomial~$T^2-a_pT+p$ of Frobenius.  Therefore
$K=\QQ(\sqrt{a_p^2-4p}\;\bigr)$, and
\[
  p = \left(\frac{a_p+\sqrt{a_p^2-4p}}{2}\right)
      \left(\frac{a_p-\sqrt{a_p^2-4p}}{2}\right)
\]
either splits or is ramified in~$K$. But we can rule out the latter
case by noting that
\[
  \text{$p$ ramified in $K$} 
  \quad\Longrightarrow\quad
  p\mid a_p^2-4p
  \quad\Longrightarrow\quad
  p\mid a_p \quad\Longrightarrow\quad a_p = 0.
\]
This contradicts the fact that~$a_p$ is odd, and hence~$p$ splits
in~$K$.
\end{proof}


\begin{proof}[Proof of Theorem~$\ref{thm:amicableCMconj}$]  
Up to $\Qbar$-isomorphism, there are~$13$ elliptic curves defined
over~$\QQ$ that have complex multiplication. For a list, see for
example~\cite[A~\S3]{ATAEC}.  There are three isomorphism classes
whose conductor~$N_E$ is a power of two:
\begin{align*}
  E:y^2 &= x^3+x, & N_E = 2^6, \\
  E:y^2 &= x^3+6x^2+x, & N_E = 2^5, \\
  E:y^2 &= x^3+4x^2+2x, & N_E = 2^8.
\end{align*}
All three of these curves have a nontrivial two-torsion point, as do
all of their~$\Qbar/\QQ$ twists, so $\#E(\FF_p)$ is even for all
$p\ge3$. Hence none of these curves admit an amicable pair;
cf.\ Remark~\ref{remark:noaliqiftors}. The remaining CM curves have
complex multiplication by a field~$\QQ(\sqrt{-D})$ with
$D\equiv3\pmod{4}$.
\par
The endomorphism ring of~$E$ is an order in the field
$K=\QQ(\sqrt{-D})$, where $D\equiv3\pmod4$, so it has the form
\[
  \End(E) \cong
  \Ocal = \ZZ + f\ZZ\left[\frac{1+\sqrt{-D}}{2}\right]
\]
for some integer~$f\ge1$, which is called the conductor of~$\Ocal$.
In particular, we have~$\Ocal^*=\{\pm1\}$, since our assumption
that~$j(E)\ne0$ excludes the case $(D,f)=(3,1)$.
\par
The theory of complex multiplication says that there is a
Gr\"ossen\-charac\-ter~$\psi_E$ such that for every prime ideal~$\gp$
of~$\Ocal_K$ of residue characteristic~$p\ge5$ at which~$E$ has
good reduction, we have
\begin{itemize}
\setlength{\itemsep}{5pt}
\item[(i)]
$\psi_E(\gp)\in\Ocal$ with $\psi_E(\gp)\Ocal_K = \gp$.
\item[(ii)]
$\#\Etilde_\gp(\FF_\gp) = \Norm_{K/\QQ}(\gp)+1-\Trace\bigl(\psi_E(\gp)\bigr)$.
\end{itemize}
See, for example,~\cite[Proposition~4.1]{RubinSilverberg}
or~\cite[II~\S9]{ATAEC}. (Note that
our assumption that~$\gp$ has residue characteristic~$p\ge5$ implies
that~$p$ does not divide the conductor of~$\Ocal$, since our assmption
that~$\Ocal$ has class number one implies that the conductor
of~$\Ocal$ is at most~$3$.)
\par
We are given that $p\ge5$ and that $\#\Etilde_p(\FF_p)=q$ is prime. 
It follows from Lemma~\ref{lemma:psplits} that~$p$
splits in~$K$, say
\[
  p\Ocal_K = \gp\bar\gp.
\]
Then $\FF_p=\FF_\gp$, so
\begin{equation}
  \label{eqn:qfactored1}
  q = \#\Etilde_p(\FF_p) = \#\Etilde_\gp(\FF_\gp)
  = \Norm_{K/\QQ}\bigl(1-\psi_E(\gp)\bigr).
\end{equation}
Notice that this implies, in particular, that~$q$ splits in~$K$.
So writing $q\Ocal_K=\gq\bar\gq$, we have  
\begin{equation}
  \label{eqn:qfactored2}
  q = \Norm_{K/\QQ}\bigl(\psi_E(\gq)\bigr).
\end{equation}
Comparing~\eqref{eqn:qfactored1} and~\eqref{eqn:qfactored2},
and using the fact that~$\psi_E(\gp)$ and~$\psi_E(\gq)$ are in~$\Ocal$,
we see that there is a unit~$u\in\Ocal^*$ such that either
\begin{equation}
  \label{eqn:psiEgqgp}
  \psi_E(\gq) = u\bigl(1-\psi_E(\gp)\bigr)
  \qquad\text{or}\qquad
  \psi_E(\gq) = u\overline{\bigl(1-\psi_E(\gp)\bigr)}.
\end{equation}
(This follows from the fact that the factorization of the
ideal~$q\Ocal$ is unique, up to switching the factors.)
\par
As noted earlier, we have~$\Ocal^*=\{\pm1\}$, so
\begin{align*}
  \Trace\bigl(\psi_E(\gq)\bigr)
  &= \pm \Trace\bigl(1-\psi_E(\gp)\bigr)
    &&\text{from~\eqref{eqn:psiEgqgp} with $u=\pm1$,} \\
  &= \pm \left(2 - \Trace\bigl(\psi_E(\gp)\bigr)\right) 
    &&\text{linearity,}\\
  &= \pm \bigl(2 - (p+1-q)\bigr)
    &&\text{since $\#\Etilde_p(\FF_p)=q$,} \\
  &= \pm (q+1-p).
\end{align*}
Hence
\[
  \#\Etilde_q(\FF_q)
  = \#\Etilde_\gq(\FF_\gq)
  = q + 1 - \Trace\bigl(\psi_E(\gq)\bigr)
  = q + 1 \pm (q+1-p).
\]
This completes the proof of Theorem~\ref{thm:amicableCMconj}.
\end{proof}

\begin{proof}[Proof of Corollary~$\ref{corollary:CMhasnoaliqge3}$]
Let $(p_1,p_2,\ldots,p_\ell)$ be an aliquot cycle of length $\ell\ge3$
for~$E/\QQ$ such that all $p_i\ge5$.  Since the primes in the cycle
are distinct, Theorem~$\ref{thm:amicableCMconj}$ tells us that
\[
  p_i = 2p_{i-1}+2-p_{i-2}
  \quad\text{for $3\le i\le \ell$.}
\]
Further, since the term in the aliquot sequence following~$p_\ell$
is~$p_1$, we have
\begin{equation}
  \label{eqn:p12pell}
  p_1 = 2p_\ell + 2 - p_{\ell-1}.
\end{equation}
\par
Consider the linear recursion 
\[
  A_1=p,\quad A_2=q,\quad   A_i = 2A_{i-1}+2-A_{i-2}\quad\text{for $i\ge3$.}
\]
A simple calcuation shows that the general term of this recursion is
given by the formula
\begin{equation}
  \label{eqn:Airecursion}
  A_i = (i-1)q - (i-2)p + (i-1)(i-2).
\end{equation}
Hence the right-hand side of~\eqref{eqn:p12pell}, which corresponds
to~$p_{\ell+1}$, is equal to
\[
  \ell p_2 - (\ell-1) p_1 + \ell(\ell-1).
\]
Equating this with~$p_1$, rearranging terms, and dividing
by~$\ell$, yields
\begin{equation*}
  p_1 = p_2 + \ell - 1.
\end{equation*}
\par
The same argument applied to the aliquot cycle 
\[
  (p_i,p_{i+1},\cdots,p_\ell,p_1,p_2,\ldots,p_{i-1})
\]
obtained by cyclically permuting the terms in the original cycle
yields
\[
  p_i = p_{i+1} + \ell - 1
  \quad\text{for all $1\le i\le \ell$,}
\]
where we set $p_{\ell+1}=p_1$. 
Since $\ell>1$, this shows that $p_{i}>p_{i+1}$ (strict inequality). Hence
\[
  p_1 > p_2 > p_3 > \cdots > p_\ell > p_{\ell+1}=p_1.
\]
This contradiction completes the proof of
Corollary~\ref{corollary:CMhasnoaliqge3}.
\end{proof}

We now use Theorem~\ref{thm:amicableCMconj} to give an heuristic
justification for the following conjecture.

\begin{conjecture}
\label{conj:CMamicable}
Let $E/\QQ$ be an elliptic curve with with complex multiplication, and
assume that $j(E)\ne0$.  Define counting functions
\begin{align*}
   \Ncal_E(X) &= \#\bigl\{\text{primes $p\le X$ such that $\#\Etilde_p(\FF_p)$
       is prime}\bigr\}, \\[1\jot]
  \Qcal_E(X)
   &= \#\bigl\{\text{amicable pairs $(p,q)$ for $E/\QQ$ 
   with $p<q$ and  $p\le X$}\bigr\}.
\end{align*}
Then either $\Ncal_E(X)$ is bounded, or else
\[
  \lim_{X\to\infty} \frac{\Qcal_E(X)}{\Ncal_E(X)} = \frac{1}{4}.
\]
\end{conjecture}

We note that Conjecture~\ref{conj:kobzy} says that if $\Ncal_E(X)$ is
unbounded, then it is asymptotic to $C_{E/\QQ}X/(\log X)^2$. So the
combination of Conjectures~\ref{conj:kobzy} and~\ref{conj:CMamicable}
gives a strengthened version of the~CM part of
Conjecture~\ref{conj:amicablepair}.

Our justification for Conjecture~\ref{conj:CMamicable} is to
observe that Theorem~\ref{thm:amicableCMconj} says that if
$\#\Etilde_p(\FF_p)=q$ is prime, then there are two possibilities for
$\#\Etilde_q(\FF_q)$, one of which is~$p$.  Experiments indicate that
each possibility occurs with equal probability, and we have no
theoretical reasons for expecting otherwise, so we will accept the
hypothesis that
\[
  \Prob\bigl(\#\Etilde_q(\FF_q)=p \bigm| 
         \text{$\#\Etilde_p(\FF_p)=q$ is prime} \bigr)
  = \frac{1}{2}.
\]
Further, if we assume Conjecture~\ref{conj:kobzy}, then
\[
  \Prob\bigl(\text{$\#\Etilde_p(\FF_p)$ is prime}\bigm| p\le X \bigr)
  \sim \frac{\Ncal_E(X)}{\pi(X)}.
\]
Combining these estimates yields
\begin{align*}
  \#\bigl\{p\le X &{} : 
    \text{$\#\Etilde_p(\FF_p)=q$ is prime and $\#\Etilde_q(\FF_q)=p$}\bigr\} \\
  &\approx \sum_{p\le X}
    \Prob\bigl(\#\Etilde_q(\FF_q)=p \quad\text{and}\quad
         \text{$\#\Etilde_p(\FF_p)=q$ is prime} \bigr) \\
  &\approx \smash[b]{\sum_{p\le X}}
    \Prob\bigl(\#\Etilde_q(\FF_q)=p \bigm|
         \text{$\#\Etilde_p(\FF_p)=q$ is prime} \bigr)  \\
  &\omit\hfill$\displaystyle
       {}\times\Prob\bigl(\text{$\#\Etilde_p(\FF_p)$ is prime} \bigr)$  \\
  &\approx \sum_{p\le X}
    \frac{1}{2}\cdot  \frac{\Ncal_E(X)}{\pi(X)} \\
  & = \frac{\Ncal_E(X)}{2}.
\end{align*}
Finally, we need to divide by~$2$, because~$\Qcal_E(X)$ only counts
amicable pairs~$(p,q)$ that are normalized to satisfy~$p<q$.

\section{Amicable pairs for CM curves with $j=0$} 
\label{section:amicableCMj0}

In this section we study elliptic curves having $j$-invariant zero.
The analysis of amicable pairs on these curves is significantly more
complicated than on all other CM elliptic curves, due primarily to the
extra units in the endomorphism ring.  In particular, experiments
described in Section~\ref{section:amicableCMexper} suggest that the
limiting value of $\Qcal_E(X)/\Ncal_E(X)$ for the curve $y^2=x^3+k$
varies for different values of~$k$; see
Conjecture~\ref{conjecture:N1density}.

We continue with the Gr\"ossen\-charac\-ter notation from the previous
section and set some additional notation that will remain in effect
for this section.  We let
\[
  \o=\frac{1+\sqrt{-3}}{2},\qquad
  K=\QQ(\sqrt{-3}),\qquad \Ocal_K=\ZZ[\o],
\]
so~$\o$ is a primitive sixth root of unity.  We note that the
unit group $(\Ocal_K/3\Ocal_K)^*$ is a group of order~$6$, and that the
natural map
\[
  \bfmu_6 = \Ocal_K^* \xrightarrow{\;\sim\;}(\Ocal_K/3\Ocal_K)^*
\]
is an isomorphism. 
Further, for any prime ideal~$\gp$ of~$\Ocal_K$ that is relatively
prime to~$3$ and any~$\a\in\Ocal_K\setminus\gp$, we recall that the
sextic residue symbol~$\SS{\a}{\gp}$ is defined by the conditions
\[
  \SS{\a}{\gp}\in\bfmu_6
  \qquad\text{and}\qquad
  \SS{\a}{\gp} \equiv \a^{\frac16(\Norm_{K/\QQ}\gp-1)} \pmod\gp.
\]

\begin{theorem}
\label{theorem:j0charctr}
Let $k\in\ZZ$ be a nonzero integer, let~$E/\QQ$ be the elliptic curve
\[
  E:y^2=x^3+k,
\] 
so~$E$ has~CM by~$\Ocal_K$, and let~$\psi_E$ be the
Gr\"ossen\-charac\-ter associated to~$E$.  Suppose that~$p\ge5$ and~$q\ge5$
are primes of good reduction for~$E$ such that
\[
  \#\Etilde_p(\FF_p)=q.
\]
\vspace{-10pt}
\begin{parts}
\Part{(a)}
The prime~$p$  splits in~$K$, say $p\Ocal_K=\gp\bar\gp$, and satisfies
\[
  \psi_E(\gp)\bigl(1-\psi_E(\gp)\bigr) \equiv1\pmod{3\Ocal_K}.
\]
\Part{(b)}
The ideal defined by
\[
  \gq = \bigl(1-\psi_E(\gp)\bigr)\Ocal_K
  \quad\text{satisfies}\quad
  q\Ocal_K = \gq\bar\gq.
\]
In particular, the prime~$q$ splits in~$K$.
\Part{(c)}
The values of the Gr\"ossencharacter at~$\gp$ and~$\gq$ are related by
\begin{equation}
  \label{eqn:trgqtrgp}
  1-\psi_E(\gp) = \SS{4k}{\gp}\SS{4k}{\gq}\psi_E(\gq).
\end{equation}
\Part{(d)}
Let $\e\in\{\pm1\}$. 
Then the trace $a_q(E)=q+1-\#\Etilde_q(\FF_q)$ satisfies
\begin{equation}
  \label{eqn:pqrecipj0}
  a_q(E)=\e(q+1-p)
  \quad\Longleftrightarrow\quad  
  \SS{4k}{\gp}\SS{4k}{\gq} = \e.
\end{equation}
\end{parts}
\end{theorem}

\begin{remark}
The expressions in~\textup{(c)} and~\textup{(d)} appear naturally in
the course of proving Theorem~\ref{theorem:j0charctr}, but we note that
they may be simplified using Proposition~$\ref{proposition:kpworw5}$,
which says that $\SS{4}{\gp}\SS{4}{\gq}=1$. This allows us to
rewrite~\eqref{eqn:trgqtrgp} and~\eqref{eqn:pqrecipj0} as
\begin{gather}
  \label{eqn:trgqtrgpalt}
  \tag{$\ref{eqn:trgqtrgp}'$}
  1-\psi_E(\gp) = \SS{k}{\gp}\SS{k}{\gq}\psi_E(\gq),  \\
  \label{eqn:pqrecipj0alt}
  \tag{$\ref{eqn:pqrecipj0}'$}
  a_q(E)=\pm(q+1-p)
  \quad\Longleftrightarrow\quad  
  \SS{k}{\gp}\SS{k}{\gq} = \pm1.
\end{gather}
\end{remark}

\begin{proof}
The fact that~$p$ splits in~$\Ocal_K$ follows from
Lemma~\ref{lemma:psplits}, which proves the first part of~(a). Next,
as noted during the proof of Theorem~\ref{thm:amicableCMconj}, the
Gr\"ossen\-charac\-ter of a CM elliptic curve satisfies
\[
  \#\Etilde_p(\FF_p) = \Norm_{K/\QQ}\bigl(\psi_E(\gp)\bigr) 
            + 1 - \Trace_{K/\QQ}\bigl(\psi_E(\gp)\bigr).
\]
Using the given value  $q=\#\Etilde_p(\FF_p)$, this can be
written as
\[
  q = \Norm_{K/\QQ}\bigl(1-\psi_E(\gp)\bigr).
\]
Hence~$\gq=\bigl(1-\psi_E(\gq)\bigr)\Ocal_K$ satisfies
$q\Ocal_K=\gq\bar\gq$, which proves~(b). 

Further, both~$\psi_E(\gp)$ and $1-\psi_E(\gp)$ have norms
that are relatively prime to~$3$. This implies first that
$\psi_E(\gp)\equiv\o^j\pmod{3}$ for some~$j\in\ZZ$, and second
that~$j$ is odd, since otherwise~$1-\o^j$ would be divisible
by~$\sqrt{-3}$.  On the other hand, for any odd value of~$j$ it is
easy to check that
\[
  (1-\o^j)\o^j \equiv 1 \pmod{3\Ocal_K},
\]
so we find that
\begin{equation}
  \label{eqn:1psiEgppsiEgp}
  \psi_E(\gp)\bigl(1-\psi_E(\gp)\bigr) \equiv 1  \pmod{3\Ocal_K}.
\end{equation}
This proves the second assertion in~(a).

For the proof of~(c) we  use the explicit formula for the
Gr\"ossen\-charac\-ter of curves of the form~$y^2=x^3+k$ in terms of
sextic residue symbols.  This formula says that $\psi_E(\gp) =
-\SS{4k}{\gp}^{-1}\pi$, where the generator~$\pi$ is a primary
generator for~$\gp$, i.e., $\pi\equiv2\pmod{3\Ocal_K}$.  (See
\cite[Chapter~18, Theorem~4 and Section~7]{IrelandRosen}
or~\cite[Proposition~4.1]{RubinSilverberg}.)  Reducing this formula
for~$\psi_E$ modulo~$3$ and applying it to both of the
primes~$\gp$ and~$\gq$, we obtain
\begin{equation}
  \label{eqn:psiEpqSS}
  \psi_E(\gp) \equiv \SS{4k}{\gp}^{-1} \pmod{3\Ocal_K}
  \quad\text{and}\quad
  \psi_E(\gq) \equiv \SS{4k}{\gq}^{-1} \pmod{3\Ocal_K}.
\end{equation}

By definition, the ideal~$\gq$ is generated by~$1-\psi_E(\gp)$.  On
the other hand, the Gr\"ossen\-charac\-ter has the property
that~$\psi_E(\gq)$ generates the ideal~$\gq$.  It follows that there
is a unit~$u\in\Ocal_K^*=\bfmu_6$ such that
\text{$1-\psi_E(\gp)=u\psi_E(\gq)$}.  Using~\eqref{eqn:1psiEgppsiEgp}
and~\eqref{eqn:psiEpqSS}, we find that
\[
  u = \frac{1-\psi_E(\gp)}{\psi_E(\gq)}
  \equiv \frac{1}{\psi_E(\gp)\psi_E(\gq)}
  \equiv \SS{4k}{\gp} \SS{4k}{\gq}
  \pmod{3\Ocal_K}.
\]
Since a sixth root of unity is determined by its residue modulo~$3$, 
this last congruence is an equality, which completes the 
proof of~(c).
\par
Using the defining property of the Gr\"ossen\-charac\-ter and
formula~\eqref{eqn:trgqtrgp} from~(c), we have
\[
  a_q(E) = \Trace_{K/\QQ}\bigl(\psi_E(\gq)\bigr)
  = \Trace_{K/\QQ}\left(\SS{4k}{\gp}^{-1} \SS{4k}{\gq}^{-1}
       \bigl(1-\psi_E(\gp)\bigr)\right).
\]
Similarly, using the assumption that~$\#\Etilde_p(\FF_p)=q$, we
find that
\[
  \Trace_{E/\QQ}\bigl(1-\psi_E(\gp)\bigr)
  = 2 - \Trace_{E/\QQ}\bigl(\psi_E(\gp)\bigr)
  = 2 - (p+1-q)
  = q + 1 - p.
\]
Hence for~$\e\in\{\pm1\}$, we have
\begin{multline*}
  a_q(E) = \e(q+1-p) 
  \quad\Longleftrightarrow\quad\\*
  \Trace_{K/\QQ}\left(\e\SS{4k}{\gp}^{-1} \SS{4k}{\gq}^{-1}
       \bigl(1-\psi_E(\gp)\bigr)\right)
  = \Trace_{E/\QQ}\bigl(1-\psi_E(\gp)\bigr).
\end{multline*}
We now use the following lemma, which may be applied because the
quantity $\Norm_{E/\QQ}\bigl(1-\psi_E(\gp)\bigr)=q$ is neither a
square nor~$3$~times a square. The lemma allows us to conclude that
\[
  a_q(E) = \e(q+1-p) 
  \quad\Longleftrightarrow\quad
  \e\SS{4k}{\gp}^{-1} \SS{4k}{\gq}^{-1}=1,
\]
 which completes the proof of~(e).
\end{proof}

\begin{lemma}
\label{lemma:equaltraces}
Let $\a\in\Ocal_K$ have the property the~$\Norm_{K/\QQ}(\a)$ is neither a
square nor~$3$ times a square.
Then
\[
  \Trace_{K/\QQ}(\z\a)=\Trace_{K/\QQ}(\a)\quad\text{with $\z\in\bfmu_6$}
  \quad\Longleftrightarrow\quad  
  \z = 1.
\]
\end{lemma}
\begin{proof}
We have
\begin{align*}
  \Trace_{K/\QQ}(\z\a)=\Trace_{K/\QQ}(\a)
  &\quad\Longleftrightarrow\quad  
  \Trace_{K/\QQ}\bigl((\z-1)\a\bigr) = 0 \\
  &\quad\Longleftrightarrow\quad  
  (\z-1)\a = c\sqrt{-3}~\text{for some $c\in\ZZ$,} \\
  &\quad\Longleftrightarrow\quad  
  \text{$\z=1$ or $ \a = c\dfrac{\sqrt{-3}}{\z-1}$.}
\end{align*}
(Note that $c$ is in $\ZZ$ because $\z$ and~$\a$ are
in~$\Ocal_K=\ZZ[\o]$.) Suppose that~$\z\ne1$. We observe that as~$\z$
ranges over~$\bfmu_6\setminus\{1\}$, the quantity~$\sqrt{-3}/(\z-1)$
takes on the five values
\[
  \left\{2-\o,1-\o,\frac12-\o,-\o,-1-\o\right\}.
\]
The norms of these five numbers form the set~$\{1,3,\frac34\}$, so the
norm of~$\a$ would have the form~$c^2$,~$3c^2$, or~$3(c/2)^2$,
contradicting the assumption on~$\Norm_{K/\QQ}(\a)$.
\end{proof}

We can use Theorem~\ref{theorem:j0charctr} to show that for some
curves with~$j(E)=0$, the conclusion of
Theorem~\ref{thm:amicableCMconj} is true, i.e., there are only two
possible values for~$\#\Etilde_q(\FF_q)$.

\begin{corollary}
\label{cor:2d3j0}
Let $d\in\ZZ$ be a nonzero integer, and let~$E$ be the
elliptic curve $E:y^2=x^3+2d^3$.  Let $p$
be a prime with $p\notdivide 6d$ such that $q=\#\Etilde_p(\FF_p)$ is
also prime and satisfies $q\notdivide 6d$. Then
\[
  \#\Etilde_q(\FF_q)=p
  \qquad\text{or}\qquad
  \#\Etilde_q(\FF_q)=2q+2-p.
\]
\end{corollary}
\begin{proof}
Using notation from Theorem~\ref{theorem:j0charctr}, we have $k=2d^3$, so
\[
  \SS{4k}{\gp}   = \SS{2d}{\gp}^3 = \pm1
  \quad\text{and}\quad
  \SS{4k}{\gq}   = \SS{2d}{\gq}^3 = \pm1.
\]
It follows from Theorem~\ref{theorem:j0charctr}(d) that
$a_q(E)=\pm(q+1-p)$.
\end{proof}

We next prove two useful facts.

\begin{proposition}
\label{proposition:kpworw5}
Let~$k$,~$E$,~$p$,~$q$,~$\gp$, and~$\gq$ be  as in the statement of
Theorem~$\ref{theorem:j0charctr}$. 
\begin{parts}
\Part{(a)}
$\displaystyle
  \SS{k}{\gp} = \o\quad\text{or}\quad \o^5.
$
\par\vspace{2\jot}
\Part{(b)}
$\displaystyle
  \SS{2}{\gp}\SS{2}{\gq} = \LS{2}{p}_\QQ\LS{2}{q}_\QQ,
$
so in particular,
$\displaystyle
  \SS{2}{\gp}\SS{2}{\gq} =\pm1.
$
\end{parts}
\textup(In~\textup{(b)}, $\LS{\;\cdot\;}{\cdot}_\QQ$ denotes the usual
quadratic residue symbol in~$\ZZ$.\textup)

\end{proposition}
\begin{proof}
(a)\enspace
If $k$ is a square modulo~$\gp$, then~$\Etilde_\gp(\FF_\gp)$ has
a nontrivial $3$-torsion point, so~$\#\Etilde_\gp(\FF_\gp)$ cannot be
prime. Similarly, if~$k$ is a cube modulo~$\gp$,
then~$\Etilde_\gp(\FF_\gp)$ has a nontrivial $2$-torsion point,
so again $\#\Etilde_\gp(\FF_\gp)$ cannot prime. Hence
\[
  \SS{k}{\gp}^3 = \LS{k}{\gp}_2 \ne 1
  \qquad\text{and}\qquad
  \SS{k}{\gp}^2 = \LS{k}{\gp}_3 \ne 1.
\]
This means that $\SS{k}{\gp}$ cannot equal~$1$,~$\o^2$,~$\o^3$, or~$\o^4$,
so it must be either~$\o$ or~$\o^5$.
\par\noindent(b)\enspace
We first note that for any~$\a,\b\in\Ocal_K$ with $\gcd(6,\b)=1$, we
have
\begin{equation}
  \label{eqn:SSabLSCS}
  \SS{\a}{\b}^{-1} = \SS{\a}{\b}^{5}
  = \SS{\a}{\b}^{3} \SS{\a}{\b}^{2}
  = \LS{\a}{\b}_2 \CS{\a}{\b}.
\end{equation}
If, in addition, $\a\in\ZZ$, then~\cite[Chapter~18, Section~7,
Lemma~2]{IrelandRosen} says that $\LS{\a}{\b}_2 =
\LS{\a}{\Norm_{K/\QQ}(\b)}_\QQ$.

In order to prove~(b), we use cubic reciprocity~\cite[Chapter~9,
  Section~3]{IrelandRosen}. We recall that an element $\a\in\Ocal_K$
is said to be \emph{primary} if
$\a\equiv2\pmod{3\Ocal_K}$. Since~$\psi_E(\gp)$ is relatively prime
to~$3$, there is a (unique) sixth root of unity~$\z\in\bfmu_6$ such
that $\z\psi_E(\gp)$ is primary. It follows from 
Theorem~\ref{theorem:j0charctr}(a) that
$\z^{-1}\bigl(1-\psi_E(\gp)\bigr)$ is also primary, and of course, the
number~$2$ is primary. Hence cubic reciprocity yields
\begin{align}
  \label{eqn:CS2pCS2q}
  \CS{2}{\gp}\CS{2}{\gq}
  &= \CS{2}{\psi_E(\gp)}\CS{2}{1-\psi_E(\gp)} \notag\\
  &= \CS{2}{\z\psi_E(\gp)}\CS{2}{\z^{-1}(1-\psi_E(\gp)) } \notag\\
  &= \CS{\z\psi_E(\gp)}{2}\CS{\z^{-1}(1-\psi_E(\gp)) }{2} \notag\\
  &= \CS{\psi_E(\gp)(1-\psi_E(\gp))}{2}.
\end{align}
The primes $\psi_E(\gp)$ and $1-\psi_E(\gp)$ are relatively prime
to~$2$, so~$\psi_E(\gp)$ is congruent to either~$\o$ or~$1+\o$
modulo~$2$. (Note that $\Ocal_K/2\Ocal_K=\{0,1,\o,1+\o\}$.)
Hence
\[
  \psi_E(\gp)\bigl(1-\psi_E(\gp)\bigr)
  \equiv \o(1+\o) \equiv 1 \pmod{2\Ocal_K}.
\]
Substituting into~\eqref{eqn:CS2pCS2q} shows that
$\CS{2}{\gp}\CS{2}{\gq}=1$.  Using~\eqref{eqn:SSabLSCS} and its
accompanying remark, we find that
\[
  \SS{2}{\gp}^{-1}\SS{2}{\gq}^{-1}
  = \LS{2}{\gp}_2\LS{2}{\gq}_2 \CS{2}{\gp}\CS{2}{\gq}
  = \LS{2}{p}_\QQ\LS{2}{q}_\QQ,
\]
which completes the proof of~(b). 
\end{proof}

\begin{corollary}
\label{corollary:j0explicit}
Let~$E/\QQ$,~$p$, and~$q$ be as in the statement of
Theorem~$\ref{theorem:j0charctr}$. 
\begin{parts}
\Part{(a)}
There exists an integer~$A$ satisfying
\begin{equation}
  \label{eqn:A22pq2p2q}
  A^2 = \frac{2pq+2p+2q-p^2-q^2-1}{3}.
\end{equation}
\Part{(b)}
The trace
$a_q(E)=q+1-\#\Etilde_q(\FF_q)$ equals one of the following six values\textup:
\begin{equation}
  \label{eqn:p2q2pq13A2}
  \pm(q+1-p),\quad
  \frac{\pm(q+1-p)\pm3A}{2}.
\end{equation}
\end{parts}
\end{corollary}

\begin{remark}
The six possible values of~$\#\Etilde_q(\FF_q)$ described in
Corollary~\ref{corollary:j0explicit}(b) are~$\#\Etilde_q^{(d)}(\FF_q)$ for
the sextic twists of~$\Etilde_q$ corresponding to the elements of
$H^1\bigl(\Gal(\bar\FF_q/\FF_q),\Aut(\Etilde_q)\bigr) \cong
H^1\bigl(\Gal(\bar\FF_q/\FF_q),\boldsymbol\mu_6\bigr) \cong
\FF_q^*/(\FF_q^*)^6$.
\end{remark}

\begin{remark}
Using Corollary~\ref{corollary:j0explicit} and a case-by-case analysis,
we prove in Appendix~\ref{appendix:j0aliqtriples} that $j=0$ elliptic
curves have no aliquot cycles of length three.
\end{remark}

\begin{proof}
(a)\enspace
We know that $\Trace\bigl(\psi_E(\gp)\bigr)=a_p(E)$, so
writing~$\psi_E(\gp)$ as an element of~$\Ocal_K=\ZZ[\o]$, it has the form
\begin{equation}
  \label{eqn:psigptpEA32}
  \psi_E(\gp) = \frac{a_p(E)+A\sqrt{-3}}{2}
  \quad\text{for some $A\in\ZZ$.}
\end{equation}
Since we also know that~$\Norm_{K/\QQ}\bigl(\psi_E(\gp)\bigr)=p$, we find that
\begin{equation}
  \label{eqn:tpEA4}
  \frac{a_p(E)^2 + 3A^2}{4} = p.
\end{equation}
Finally, the assumption that $\#\Etilde_p(\FF_p)=q$ is equivalent
to $a_p(E)=p+1-q$. Substituting this value into~\eqref{eqn:tpEA4},
a little bit of algebra shows that~$A$ has the form specified
by~\eqref{eqn:A22pq2p2q}.
\par\noindent(b)\enspace
Applying~\eqref{eqn:trgqtrgp}
from Theorem~\ref{theorem:j0charctr}, we find that
\[
  \Trace_{K/\QQ}\bigl(\psi_E(\gq)\bigr)
  = \Trace_{K/\QQ}\bigl( \z(1-\psi_E(\gp))\bigr)
\]
for some~$\z\in\bfmu_6$.  Using the value of~$\psi_E(\gp)$
from~\eqref{eqn:psigptpEA32} with the substitution $a_p(E)=p+1-q$
yields
\[
  \Trace_{K/\QQ}\bigl(\psi_E(\gq)\bigr)
  = \Trace_{K/\QQ}
       \left( \z \left( \frac{q+1-p-A\sqrt{-3}}{2} \right) \right).
\]
Substituting in each of the six possible values~$\z\in\bfmu_6$ and 
taking the trace yields the six values listed in~\eqref{eqn:p2q2pq13A2}.
\end{proof}

\begin{definition}
Fix a non
$
  E:y^2=x^3+k.
$
We let~$\Ncal_k$ denote the set
\[
  \Ncal_k = \left\{
    \begin{tabular}{@{}l@{}}
       primes $p\ge5$ of good  reduction for~$E$\\
       such that $q=\#\Etilde_p(\FF_p)$ is also\\   
        a prime of good reduction for $E$\\
    \end{tabular}
  \right\}.
\]
(This differs slightly from our earlier notation in that we are now
excluding a few primes, but this does not affect our asymptotic formulas.)
We define a subset of~$\Ncal_k$ by
\[
  \Ncal_k^{[1]} = \bigl\{p\in\Ncal_k : a_q(E)=\pm(q+1-p) \bigr\},
\]
and we say that the primes in~$\Ncal_k^{[1]}$ are of Type~1 for~$E$.
We write~$\Ncal_k(X)$ for the number of primes in~$\Ncal_k$ that are
less than~$X$, and similarly for~$\Ncal_k^{[1]}(X)$.
\end{definition}

Only Type~1 primes can be amicable, and based on experiments, we
expect that about half of the Type~1 primes will be
part of an amicable pair. Let
\[
  \Qcal_k(X) = \#\bigl\{p < X :
   \text{$p<q$ and $(p,q)$ is an amicable pair for $E$}\bigr\},
\]
i.e., $\Qcal_k(X)$ is the number of normalized amicable pairs~$(p,q)$
on~$E$ with~$p<X$. Then we have the following conjecture, where the
conjectured limit is~$\frac14$, rather than~$\frac12$,
because~$\Qcal_k(X)$ counts amicable pairs~$(p,q)$ with~$p<q$,
while $\Ncal_k^{[1]}(X)$ counts both~$(p,q)$ and~$(q,p)$.

\begin{conjecture}
\label{conj:QoverNeq12}
With notation as above, the proportion of Type~$1$ primes that 
are part of a normalized amicable pair is given by
\[
  \lim_{X\to\infty} \frac{\Qcal_k(X)}{\Ncal_k^{[1]}(X)} = \frac{1}{4}.
\]
\end{conjecture}

Thus in order to understand the distribution of amicable pairs on~$E$,
we need to study the density of the Type~1 primes in~$\Ncal_k$.

\begin{remark}
According to Corollary~\ref{corollary:j0explicit}, there are six
possible values for~$a_q(E)$, two of which give Type~1 primes, so one
might expect~$\Ncal_k^{[1]}$ to have density~$\frac13$
inside~$\Ncal_k$. This turns out not to be the case.  At the extreme
end, Corollary~\ref{cor:2d3j0} says that
$\Ncal_{2d^3}^{[1]}=\Ncal_{2d^3}$ for any nonzero $d\in\ZZ$.  The rest
of this section is devoted to developing tools for calculating a
conjectural value for $\lim_{X\to\infty}\Ncal_k^{[1]}(X)/\Ncal_k(X)$.
This value depends on~$k$ in quite a complicated way;
see Conjecture~\ref{conjecture:N1kN1M1kM1}.  For precise formulas when~$k$ is
prime, see Conjecture~\ref{conjecture:N1density}, which says that the
limit should equal~$\frac13+R(k)$, where~$R(k)$ is a rational function
of~$k$ that depends on~$k$ modulo~$36$.
\end{remark}

\begin{definition}
We set the notation
\[
  n\pequiv a\pmod{m}
  \quad\Longleftrightarrow\quad
  p\equiv a\pmod{m}~\text{for every prime $p\mid n$.}
\]
Further, for any ideal~$\gK\subset\Ocal_K$ we define
\[
  \OK{\gK} = 
  \left\{\l\in\frac{\Ocal_K}{\gK} : \gcd\bigl(\l(1-\l),\gK\bigr)=1 \right\}.
\]
If~$\gK=k\Ocal_K$ is principal, we write simply~$\OK{k}$.
\par
Now let $k\in\ZZ$ satisfy $\gcd(6,k)=1$. We define a set~$M_k$ that
depends on~$k$ modulo~$4$ and on the primes dividing~$k$ modulo~$9$.
\begin{align*}
  \multispan{2}{
  \framebox{(a) $k\equiv1\pmod4$ and $k\pequiv\pm1\pmod9$}
   \hfill} \\
  \hspace*{4em}
  M_k &= \left\{ \l\in\OK{k} : \LS{\l}{k}_2=-1 \text{ and } \CS{\l}{k}\ne1
  \right\}. 
  \hspace*{4em}\\
  \multispan{2}{
  \framebox{(b) $k\equiv1\pmod4$ and $k\not\pequiv\pm1\pmod9$}
   \hfill} \\
  M_k &= \left\{ \l\in\OK{k} : \LS{\l}{k}_2=-1
  \right\}. \\
  \multispan{2}{
  \framebox{(c) $k\equiv3\pmod4$ and $k\pequiv\pm1\pmod9$}
   \hfill} \\
  M_k &= \left\{ \l\in\OK{k} : \CS{\l}{k}\ne1
  \right\}. \\
  \multispan{2}{
  \framebox{(d) $k\equiv3\pmod4$ and $k\not\pequiv\pm1\pmod9$}
   \hfill} \\
  M_k &= \OK{k}. \\
  \noalign{\noindent Further, for every $k$ we define a subset of $M_k$ by}
  M^{[1]}_k &= 
  \left\{\l\in M_k : \CS{\l(1-\l)}{k}=1\right\}.
\end{align*}
\end{definition}

\begin{remark}
It is easy to check that $k\pequiv\pm1\pmod9$ if and only if every
cube root of unity in~$\Ocal_K/k\Ocal_K$ is itself a cube.  For
example, suppose that $k\in\ZZ$ is prime.  If $k\equiv-1\pmod9$, then
$\Ocal_K/k\Ocal_K\cong\FF_{k^2}$ is a finite field with~$k^2$
elements, and $\bfmu_9\subset\FF_{k^2}$.  Similarly, if
$k\equiv1\pmod9$, then $\Ocal_K/k\Ocal_K\cong\FF_k\times\FF_k$, and
$\bfmu_9\subset\FF_k$. Thus in both cases, every cube root of unity
in~$\Ocal_K/k\Ocal_K$ is itself a cube.
\end{remark}

\begin{conjecture}
\label{conjecture:N1kN1M1kM1}
Let $k\in\ZZ$ be an integer satisfying $\gcd(6,k)=1$. Then
\begin{equation}
  \label{eqn:limNk1NkMk1Mk}
  \lim_{X\to\infty} \frac{\Ncal_k^{[1]}(X)}{\Ncal_k(X)}
  = \frac{\#\Mcal^{[1]}_k}{\#\Mcal_k}.
\end{equation}
\end{conjecture}

\begin{remark}
For small values of~$k$ it is not difficult to compute the
sets~$\Mcal_k$ and~$\Mcal_k^{[1]}$, thereby obtaining an explicit
(conjectural) value for the limit~\eqref{eqn:limNk1NkMk1Mk}.
Table~\ref{table:setsOKkMkMk1} gives some examples corresponding to
the four cases~(a)--(d) used to define~$\Mcal_k$, further divided
according to the value of~$k$ modulo~$3$.  (The notation~$(x.n)$ after
each value of~$k$ indicates the case $x=\text{(a),\dots,(d)}$ and the
congruence class~$k\equiv n\pmod3$.) 
\end{remark}

\begin{table}
\small
\[
  \renewcommand{\arraystretch}{1.25}
  \begin{array}{|c|c|c|c|c|} \hline
    k & \#\OK{k} & \#M_k & \#M_k^{[1]} & \#M_k^{[1]}/\#M_k \\ \hline
    37 \;\text{(a.1)} & 1225 & 408 & 144 & \frac{6}{17} = 0.3529 \\ \hline
    17 \;\text{(a.2)} & 287 & 96 & 36 & \frac{3}{8} = 0.3750 \\ \hline
    13 \;\text{(b.1)} & 121 & 60 & 20 & \frac{1}{3} = 0.3333 \\ \hline
    \phantom{0}5 \;\text{(b.2)} & 23 & 12 & 4 & \frac{1}{3} = 0.3333 \\ \hline
    19 \;\text{(c.1)} & 289 & 192 & 72 & \frac{3}{8} = 0.3750 \\ \hline
    71 \;\text{(c.2)} & 5039 & 3360 & 1152 & \frac{12}{35} = 0.3429 \\ \hline
    \phantom{0}7 \;\text{(d.1)} & 25 & 25 & 13 & \frac{13}{25} = 0.5200
                 \\ \hline    
    11 \;\text{(d.2)} & 119 & 119 & 47 & \frac{47}{119} = 0.3950 \\ \hline
  \end{array}
\]
\caption{The sets $\OK{k}$, $M_k$ and $M_k^{[1]}$}
\label{table:setsOKkMkMk1}
\end{table}

Our justification for Conjecture~\ref{conjecture:N1kN1M1kM1} uses the
following weak form of quadratic and cubic reciprocity for the
field~$\QQ(\o)$.

\begin{lemma}
\label{lemma:srinQo}
Let $k\in\ZZ$ satisfy~$\gcd(k,6)=1$, and let $\l\in\ZZ[\o]$ satisfy
$\gcd(6k,\l)=1$. 
\begin{parts}
\Part{(a)}
\textup{(Quadratic Reciprocity in $\QQ(\o)$)}
\[
  \LS{k}{\l}_2 = (-1)^{\frac{\Norm(\l)-1}{2}\cdot\frac{k-1}{2}}\LS{\l}{k}_2.
\]
\Part{(b)}
\textup{(Cubic Reciprocity in $\QQ(\o)$)}
Let~$\z\in\bfmu_3$  be the unique cube root of unity such that
\[
  \z\l\equiv\pm1\pmod{3\Ocal_K}.
\]
Then
\[
  \CS{k}{\l} = \CS{\z}{k}\CS{\l}{k}.
\]
\end{parts}
\end{lemma}
\begin{proof}
Let $\a,\b\in\ZZ[\o]$ satisfy $\gcd(\a,\b)=\gcd(\a\b,6)=1$.  We start
with the sextic reciprocity law for~$\QQ(\o)$ as stated
in~\cite[Theorem~7.10]{Lemmermeyer}. This says that if~$\a$ and~$\b$
are ``$E$-primary'' (see~\cite{Lemmermeyer} for terminology), then
\begin{equation}
  \label{eqn:qrinO1}
  \SS{\a}{\b}\SS{\b}{\a}^{-1} = 
  (-1)^{\frac{\Norm(\a)-1}{2}\cdot\frac{\Norm(\b)-1}{2}}.
\end{equation}
Let $\r=\o^2$ denote a primitive cube root of unity. Then
for~$\a\in\Ocal_K$ satisfying $\gcd(6,\a)=1$, we have by definition
\[
  \text{$\a$ is $E$-primary}
  \quad\Longleftrightarrow\quad
  \left\{
  \begin{tabular}{@{}l@{}}
    $\a\equiv\pm1\pmod3$ and\\
    $\a^3=A+B\rho$ with $A+B\equiv1\pmod4$.\\
  \end{tabular}
  \right.
\]
(This is a corrected version of~\cite[Lemma~7.9]{Lemmermeyer}, which
omits the $\a\equiv\pm1\pmod3$ condition and includes a superfluous
$3\mid B$ requirement.)\FOOTNOTE{%
We assume $\a=a+b\r$ with $3\mid b$. Note the assumption
$\gcd(\a,2)=1$ ensures that one of~$a$ or $b$ is odd.  Further
$\a^3=(a^3-3ab^2+b^3)+(3ab(a-b)\rho$, so $A+B=a^3+3a^2b-6ab^2+b^3$. We
consider three cases. (I) $2\mid b$.  Then $A+B\equiv a^3+3a^2b\equiv
a+b\pmod{4}$, since $a^2\equiv1\pmod4$. (II) $2\mid a$. Then
$A+B\equiv b^3 \equiv b\pmod4$. (III) $2\nmid ab$. Then
$A+B\equiv a+3b-6a+b\equiv a\pmod4$. Thus in all three cases, 
the condition $A+B\equiv1\pmod4$ is equivalent to the condition on~$a$
and~$b$ for~$\a$ to be~$E$-primary.
}
We note that if~$\a\equiv\pm1\pmod{3}$, then
exactly one of~$\pm\a$ is $E$-primary. 
\par
We now consider~$k\in\ZZ$ and~$\l\in\Ocal_K$ as in the statement of
the lemma. Since~$k$ is an integer and satisfies $\gcd(6,k)=1$, we
have
\begin{align*}
  \text{$k$ is $E$-primary}
  &\quad{\Longleftrightarrow}\quad
  k\equiv\pm1\pmod3\quad\text{and}\quad k^3\equiv1\pmod4\\
  &\quad{\Longleftrightarrow}\quad
  k\equiv1\pmod{4},
\end{align*}
so $(-1)^{(k-1)/2}k$ is $E$-primary.  We also note that for
any~$\a\in\Ocal_K$ satisfying~$\gcd(6,\a)=1$, Euler's formula says
that
\begin{equation}
  \label{eqn:ssminus1}
  \SS{-1}{\a} \equiv (-1)^{\frac{\Norm(\a)-1}{6}} \pmod{\a\Ocal_K},
\end{equation}
and since both sides of~\eqref{eqn:ssminus1} are sixth roots of unity,
the congruence~\eqref{eqn:ssminus1} is an equality. In particular,
\begin{equation}
  \label{eqn:SSminus1k}
  \SS{-1}{k} = (-1)^{\frac{\Norm(k)-1}{6}} = (-1)^{\frac{k^2-1}{6}} = 1.
\end{equation}
\par
It is an easy exercise to verify that there is a unique~$\z\in\bfmu_3$
such that $\z\l\equiv\pm1\pmod3$, cf.\ \cite[Chapter~9,
Proposition~9.3.5]{IrelandRosen}.\FOOTNOTE{%
Writing $\l=a+b\r$ where~$\r=\o^2$ is a primitive cube root of unity, we
have $\r\l=-b+(a-b)\r$ and $\r^2\l=b-a-a\r$. The  condition $\gcd(3,\l)=1$
ensures that at most one of~$a$,~$b$, and~$a-b$ is divisible by~$3$. 
Suppose that none of them is divisible by~$3$. Then $a\equiv \pm1\pmod3$
and $b\equiv\mp1\pmod3$ with opposite signs, so
\[
  \l \equiv \pm1\mp\r \equiv \pm(1-\r) \equiv \pm\frac{3-\sqrt{-3}}{2}
  \pmod{3}.
\]
This contradicts the assumption that $\gcd(3,\l)=1$, hence (exactly)
one of~$a$,~$b$, and~$a-b$ is divisible by~$3$.
}
Then one of $\pm\z\l$ is~$E$-primary, so we can
apply~\eqref{eqn:qrinO1} to the $E$-primary numbers
$\a=(-1)^{(k-1)/2}k$ and~$\b=\pm\z\l$.  Then~\eqref{eqn:qrinO1}
becomes
\[
  \SS{(-1)^{(k-1)/2}k}{\l}\SS{\pm\z\l}{k}^{-1} = 
  (-1)^{\frac{k^2-1}{2}\cdot\frac{\Norm(\l)-1}{2}} = 1.
\]
(The second equality comes from the fact that $k^2\equiv1\pmod4$.)
Hence
\[
  \SS{-1}{\l}^{(k-1)/2} \SS{k}{\l}  
     \SS{\pm1}{k}^{-1} \SS{\z}{k}^{-1} \SS{\l}{k}^{-1} =1.
\]
Using~\eqref{eqn:ssminus1} and~\eqref{eqn:SSminus1k} gives 
\begin{equation}
  \label{eqn:SSkl1Nl12k12}
  \SS{k}{\l} = (-1)^{\frac{\Norm(\l)-1}{2}\cdot\frac{k-1}{2}}
  \SS{\z}{k}\SS{\l}{k}.
\end{equation}
(We note in particular that the sign used to ensure that~$\z\l$
is~$E$-primary turns out to be irrelevant
because~\text{$\SS{-1}{k}=1$}.) Cubing~\eqref{eqn:SSkl1Nl12k12}
and using~$\z^3=1$ gives the quadratic reciprocity formula in~(a),
and similarly squaring~\eqref{eqn:SSkl1Nl12k12} gives the cubic
reciprocity formula in~(b).
\end{proof}

\begin{proof}[Justification for Conjecture~$\ref{conjecture:N1kN1M1kM1}$]
Let $p\in\Ncal_k$, so Theorem~\ref{theorem:j0charctr}(a) tells us
that~$p$ splits in~$\Ocal_K$, say $p\Ocal_K=\gp\bar\gp$. As in that
theorem, we let \text{$\gq=\bigl(1-\psi_E(\gp)\bigr)\Ocal_K$}. Then squaring
Theorem~\ref{theorem:j0charctr}(d) yields
\[
  p\in\Ncal_k^{[1]}
  \quad\Longleftrightarrow\quad
  \CS{4k}{\gp}\CS{4k}{\gq}=1.
\]
Further, Proposition~\ref{proposition:kpworw5} 
implies that $\CS{2}{\gp}\CS{2}{\gq}=1$, so we find that
\[
  p\in\Ncal_k^{[1]}
  \quad\Longleftrightarrow\quad
  \CS{k}{\gp}\CS{k}{\gq}=1.
\]
The (prime) ideals~$\gp$ and~$\gq$ are generated, respectively,
by the elements~$\psi_E(\gp)$ and $1-\psi_E(\gp)$, and
Theorem~\ref{theorem:j0charctr}(a) says that these elements satisfy
\begin{equation}
  \label{eqn:thm18bmod3}
  \psi_E(\gp)\bigl(1-\psi_E(\gp)\bigr) \equiv1\pmod{3\Ocal_K}.
\end{equation}
Hence if we choose $\xi\in\bfmu_6$ to satisfy
\[
  \xi\psi_E(\gp)\equiv \pm1 \pmod{3\Ocal_K},
\]  
then~\eqref{eqn:thm18bmod3} says that we also have
\[
  \xi^{-1}\bigl(1-\psi_E(\gp)\bigr)\equiv \pm1 \pmod{3\Ocal_K}.
\]
This allow us to apply cubic reciprocity (Lemma~\ref{lemma:srinQo}(b),
or~\cite[Chapter~9, Section~3, Theorem~1]{IrelandRosen}) to
compute
\begin{align*}
  \CS{k}{\gp}\CS{k}{\gq}
  &= \CS{k}{\xi\psi_E(\gp)\Ocal_K}\CS{k}{\xi^{-1}(1-\psi_E(\gp))\Ocal_K}\\
  &= \CS{\xi\psi_E(\gp)}{k\Ocal_K}\CS{\xi^{-1}(1-\psi_E(\gp))}{k\Ocal_K}\\
  &= \CS{\psi_E(\gp)}{k\Ocal_K}\CS{1-\psi_E(\gp)}{k\Ocal_K}.
\end{align*}
Hence
\begin{equation}
  \label{eqn:pNk1iffp1pk1}
  p\in\Ncal_k^{[1]}
  \quad\Longleftrightarrow\quad
  \CS{\psi_E(\gp)\bigl(1-\psi_E(\gp)\bigr)}{k\Ocal_K}=1.
\end{equation}
\par
We now consider how the values $\psi_E(\gp)$ are distributed in
$\Ocal_K/k\Ocal_K$ as~$p$ varies in~$\Ncal_k$.  If~$p$ were chosen
completely randomly,subject  only to $p\equiv1\pmod3$, then we
might expect the values of~$\psi_E(\gp)$ to be uniformly distributed
among the congruence classes in~$\Ocal_K/k\Ocal_K$. However,
Proposition~\ref{corollary:j0explicit}(a) tells us that~$\SS{k}{\gp}$
equals either~$\o$ or~$\o^5$, i.e., it is a primitive sixth root of
unity. Equivalently,
\begin{equation}
  \label{eqn:LSkCSkconstr}
  \LS{k}{\gp}_2 = -1
  \qquad\text{and}\qquad
  \CS{k}{\gp} = \o^2\text{ or }\o^4,
\end{equation}
i.e., neither~$\LS{k}{\gp}_2$ nor~$\CS{k}{\gp}$ equals~$1$.
This gives a constraint on the values of~$\psi_E(\gp)$ for
$p\in\Ncal_k$. Discarding finitely many elements of~$\Ncal_k$, we may
assume that $p\notdivide6k$, and then reciprocity
(Lemma~\ref{lemma:srinQo}) tells us that
\[
  \LS{k}{\gp}_2 = (-1)^{\frac{p-1}{2}\cdot\frac{k-1}{2}}\LS{\psi_E(\gp)}{k}_2
  \qquad\text{and}\qquad
  \CS{k}{\gp} = \CS{\z}{k}\CS{\psi_E(\gp)}{k},
\]
where~$\z\in\bfmu_3$ satisfies $\z\psi_E(\gp)\equiv\pm1\pmod3$.  (Note
that $\Norm\bigl(\psi_E(\gp)\bigr)=p$.)  Hence the
constraints~\eqref{eqn:LSkCSkconstr} on~$\psi_E(\gp)$ from
Proposition~\ref{corollary:j0explicit}(a) become
\begin{equation}
  \label{eqn:m1p12k12SS}
   \LS{\psi_E(\gp)}{k}_2 = -(-1)^{\frac{p-1}{2}\cdot\frac{k-1}{2}}
  \qquad\text{and}\qquad
  \CS{\z}{k}\CS{\psi_E(\gp)}{k} = \o^2\text{ or }\o^4.
\end{equation}
\par
We now make the following two assumptions, which are supported by
experiments:
\begin{itemize}
\item
For $p\in\Ncal_k$, the value of $p\bmod4$ is equally likely to be~$1$ or~$3$.
\item
For $p\in\Ncal_k$, the value of~$\z$ in~\eqref{eqn:m1p12k12SS} is 
equally likely to be~$1$,~$\o^2$, or~$\o^4$.
\end{itemize}
These assumptions have the following consequences:
\begin{itemize}
\item
If $k\equiv3\pmod4$, then the first equation in~\eqref{eqn:m1p12k12SS} 
has no effect on the value of $\psi_E(\gp)\bmod k$.
\item
If $k\not\pequiv\pm1\pmod9$, i.e., if cube roots of unity
in~$\Ocal_K/k\Ocal_K$ are not necessarily cubes, then the second
equation in~\eqref{eqn:m1p12k12SS} has no effect
on the value of $\psi_E(\gp)\bmod k$.
\end{itemize}
On the other hand, if $k\equiv1\pmod4$, then the first equation
in~\eqref{eqn:m1p12k12SS} gives the constraint
$\LS{\psi_E(\gp)}{k}_2=-1$; and similarly, if $k\pequiv\pm1\pmod9$,
then the second equation in~\eqref{eqn:m1p12k12SS} imposes the
condition $\CS{\psi_E(\gp)}{k}\ne1$. Thus considering the four cases,
we see that~$\psi_E(\gp)$ is in the set~$M_k$.
Further, we note that~\eqref{eqn:pNk1iffp1pk1} says
that~$p\in\Ncal_k^{[1]}$ if and only if $\psi_E(\gp)\in\Mcal_k^{[1]}$.
Hence it is reasonable to conjecture that the density of~$\Ncal_k^{[1]}$
in~$\Ncal_k$ is given by the ratio~$\#\Mcal_k^{[1]}/\#\Mcal_k$.
\end{proof}

Conjecture~\ref{conjecture:N1kN1M1kM1} is reasonably satisfactory in
that the sets~$\Mcal_k$ and~$\Mcal_k^{[1]}$ are easy to compute for
any particular (not-too-large) value of~$k$.  In the remainder of this
section  we derive explicit formulas for~$\#\Mcal_k$
and~$\#\Mcal_k^{[1]}$ when~$k$ is prime. We do this by breaking them
up into subsets of the following sort.  For any
ideal~$\gK\subset\Ocal_K$ and any roots of unity~$\z\in\bfmu_6$ and
$\xi\in\bfmu_3$, we define
\begin{align*}
  M_\gK(\z) 
  &= \left\{\l\in\OK{\gK} : \SS{\l}{\gK}=\z \right\}\\
  &= \left\{\l\in\OK{\gK} : \LS{\l}{\gK}_2=\z^3 
          \text{ and }\CS{\l}{\gK}=\z^2 \right\},\\
  M_\gK^{[1]}(\z,\xi) &= \left\{\l\in M_\gK(\z) : \CS{\l(1-\l)}{\gK}=\xi\right\}.
\end{align*}
As before, if~$\gK=k\Ocal_K$ is principal, we write $M_k(\z)$
and~$M_k^{[1]}(\z,\xi)$. Further, if~$S\subset\bfmu_6$ is any set of
roots of unity, we write~$M_\gK(S)$ for the union of~$M_\gK(\z)$
with~$\z\in S$.  With this notation, the four cases defining~$M_k$ are
given by
\par
\begin{center}
  \renewcommand{\arraystretch}{1.5}
\begin{tabular}{rl@{\qquad}rl}
  \textup{(a)}&$M_k = M_{k}\bigl(\{\o,\o^5\}\bigr)$,
  &\textup{(b)}&$M_k = M_{k}\bigl(\{\o,\o^3,\o^5\}\bigr)$,\\
  \textup{(c)}&$M_k = M_{k}\bigl(\{\o,\o^2,\o^4,\o^5\}\bigr)$,
  &\textup{(d)}&$M_k = M_{k}(\bfmu_6)$,
\end{tabular}
\end{center}
and in all cases,~$M_k^{[1]}=M_k^{[1]}(S,1)$, where~$S\subset\bfmu_6$ is the
set for the appropriate case.
\par
We now restrict attention to the case that~$k\in\ZZ$ is a rational prime
with $\gcd(6,k)=1$. If $k\equiv2\pmod3$, so~$k$ is inert in~$K$,
then the computation of~$M_k$ and~$M_k^{[1]}$
takes place in the field $\Ocal_K/k\Ocal_K\cong\FF_{k^2}$
with~$k^2$~elements.  On the other hand, if $k\equiv1\pmod3$,
so~$k$ splits as $k\Ocal_K=\gK\bar\gK$, then
\[
  \frac{\Ocal_K}{k\Ocal_K}
  \cong\frac{\Ocal_K}{\gK}\times\frac{\Ocal_K}{\bar\gK}
  \cong\FF_k\times\FF_k.
\]
In this case a condition such as $\CS{\l}{k}\ne1$ becomes more
complicated, since there are many ways for the product
$\CS{\l}{\gK}\CS{\l}{\bar\gK}$ to be different from~$1$.

\begin{proposition}
\label{propostion:MkStable}
Let $k\ge5$ be a rational prime.  The following table gives the values
of~$\#M_k(S)$ for various subsets $S\subset\bfmu_6$, divided into
cases according to whether~$k$ is split or inert
in~$K=\QQ(\sqrt{-3}\,)$.
\begin{center}
\renewcommand{\arraystretch}{1.25}
\begin{tabular}{|c|c||c|c|} \cline{3-4}
  \multicolumn{2}{c||}{}
  &$k\equiv1\pmod3$&$k\equiv2\pmod3$\\ \hline\hline
  \textup{(a)} & $\#M_{k}\bigl(\{\o,\o^5\}\bigr)$
     & $\frac13(k-1)(k-3)$ & $\frac13(k^2-1)$ \\ \hline
  \textup{(b)} & $\#M_{k}\bigl(\{\o,\o^3,\o^5\}\bigr)$
     & $\frac12(k-1)(k-3)$ & $\frac12(k^2-1)$ \\ \hline
  \textup{(c)} & $\#M_{k}\bigl(\{\o,\o^2,\o^4,\o^5\}\bigr)$
     & $\frac{2}{3}(k-1)(k-3)$ & $\frac23(k^2-1)$ \\ \hline
  \textup{(d)} & $\#M_{k}(\bfmu_6)$
     & $(k-2)^2$ & $k^2-2$ \\ \hline
\end{tabular}
\end{center}
\end{proposition}
\begin{proof}
Suppose first that~$k$ is inert, so~$\Ocal_K/k\Ocal_K\cong\FF_{k^2}$. 
Then~$\#M_k(\bfmu_6)$ simply counts the~$\l\in\FF_{k^2}$ such that~$\l$ 
and~$1-\l$ are units, so is equal to~$k^2-2$. Next, 
$\#M_{k}\bigl(\{\o,\o^3,\o^5\}\bigr)$ counts the quadratic non-residues 
in~$\FF_{k^2}$, of which there are~$\frac12(k^2-1)$. (Here the 
condition that \text{$1-\l$} be a unit is irrelevant, since~$1$ is a 
quadratic residue).  Similarly, 
$\#M_{k}\bigl(\{\o,\o^2,\o^4,\o^5\}\bigr)$ counts cubic non-residues 
in~$\FF_{k^2}$, of which there are~$\frac23(k^2-1)$.  Finally, 
$\#M_{k}\bigl(\{\o,\o^5\}\bigr)$ counts the elements that are neither 
quadratic nor cubic residues, of which there are~$\frac13(k^2-1)$. 
\par 
Next suppose that~$k$ splits, so 
$\Ocal_K/k\Ocal_K\cong\FF_k\times\FF_k$.  Then~$\#M_k(\bfmu_6)$ counts 
$(a,b)\in\FF_k^2$ with neither~$a$ nor~$b$ equal to~$0$ or~$1$.  This 
gives \text{$k-2$} possibilities for each of~$a$ and~$b$, 
so~$\#M_k(\bfmu_6)=(k-2)^2$.   
\par 
The required calculations for each of the remaining cases have much in
common, so we will only illustrate the case of
$M_k\bigl(\{\o,\o^3,\o^5\}\bigl)$.  Exactly~$\frac12$ of
invertible~$(a,b)$ are quadratic non-residues.  Therefore, there
are~$\frac12 (k-1)^2$ such elements.  Of these, there are~$\frac12
(k-1)$ of the form~$(1,b)$ and~$\frac12 (k-1)$ of the form~$(a,1)$.
The set~$M_{k}\bigl(\{\o,\o^3,\o^5\}\bigr)$ counts invertible~$(a,b)$
that are quadratic non-residues having $a \neq 1$ and $b \neq 1$.
Therefore
\[ 
\#M_k\bigl(\{\o,\o^3,\o^5\}\bigr) = \frac12(k-1)^2 - 2\left( \frac12 (k-1) \right)  
= \frac12 (k-1)(k-3). 
\] 
(Note that $(1,1)$ is a quadratic residue, so the invertible
non-residues of the form $(a,1)$ and $(1,b)$ are disjoint.)  A similar
argument applies to the two remaining cases, where we rely on the fact
that invertible elements of $\FF_k \times \FF_k$ and of $\FF_k$ fall
evenly into the six sextic residue classes.
\end{proof}

The table in Proposition~\ref{propostion:MkStable} gives the value
of~$\#M_k(S)$ for the four subsets~$S\subset\bfmu_6$ that appear in
Conjecture~\ref{conjecture:N1kN1M1kM1}.  It remains to construct a
similar table for the values of~$\#M_k^{[1]}(S)$. It turns out that
these values can be expressed in terms of the number of points on a
certain curve of genus four over various finite fields. We begin with
a description of the curve that we need, after which we count points
in order to compute the desired values.

\begin{proposition}
\label{proposition:ECDcurves}
Let~$\FF$ be a perfect field of characteristic not equal to~$2$
or~$3$.  For $\k\in\FF^*$ we define~$E^{(\k)}$ to be
the elliptic curve
\[
  E^{(\k)} : y^2 = x^3 + \k,  
\]
and for $\g,\d\in\FF^*$ we define $C_6^{(\g,\d)}$
to be a smooth projective model for the algebraic
curve given by the affine equation
\[
  C_6^{(\g,\d)} :  \g z^6(1-\g z^6) = \d x^3.
\]
\vspace{-10pt}
\begin{parts}
\Part{(a)}
The curve $C_6^{(\g,\d)}$ has genus four.
\Part{(b)}
There are finite maps from $C_6^{(\g,\d)}$ to curves of the form~$E^{(\k)}$
given by the following formulas\textup:
\begin{align*}
   C_6^{(\g,\d)} &\longrightarrow E^{(16\d^2)},
  & (x,z) &\longmapsto (-4\d x,8\g\d z^6-4\d),  \\
   C_6^{(\g,\d)} &\longrightarrow E^{(4\g^3\d^4)},
  & (x,z) &\longmapsto 
    \left(\dfrac{\d^2 x^2}{z^6},\g^2\d^2 z^3+\dfrac{\g\d^2}{z^3}\right), \\
   C_6^{(\g,\d)} &\longrightarrow E^{(\g^5\d^2)}, 
  & (x,z) &\longmapsto 
     \left(\dfrac{\g\d x}{z^4},\dfrac{\g^2\d}{z^3}\right), \\
   C_6^{(\g,\d)} &\longrightarrow E^{(-\g\d^2)}, 
  & (x,z) &\longmapsto 
     \left(-\dfrac{\d x}{z^2},\g\d z^3\right).
\end{align*}
\Part{(c)}
The maps in~\textup{(b)} are independent, hence they induce an isogeny
\[
  E^{(16\d^2)}\times E^{(4\g^3\d^4)}\times E^{(\g^5\d^2)}\times E^{(-\g\d^2)}
  \longrightarrow  J_6^{(\g,\d)}
  \;\stackrel{\textup{def}}{=}\; \Jac(C_6^{(\g,\d)}) .
\]
\Part{(d)}
For any prime~$\ell$ different from the characteristic of~$\FF$, we have
isomorphisms of $\Gal(\bar\FF/\FF)$-modules,
\begin{align*}
  H^1_\et\bigl({C_{6/\FF}^{(\g,\d)}},\QQ_\ell\bigr)
  &\cong H^1_\et\bigl({J_{6/\FF}^{(\g,\d)}},\QQ_\ell\bigr) \\
  &\cong H^1_\et\bigl(E^{(16\d^2)}_{/\FF},\QQ_\ell\bigr) \times
    H^1_\et\bigl(E^{(4\g^3\d^4)}_{/\FF},\QQ_\ell\bigr)  \\
  &\qquad{}
    \times H^1_\et\bigl(E^{(\g^5\d^2)}_{/\FF},\QQ_\ell\bigr) \times
    H^1_\et\bigl(E^{(-\g\d^2)}_{/\FF},\QQ_\ell\bigr).
\end{align*}
\end{parts}
\end{proposition}
\begin{proof}
(a)\enspace
All of the~$C_6^{(\g,\d)}$ curves are geometrically isomorphic, 
so it suffices to calculate the genus of~$C_6^{(1,1)}$, which for
convenience we denote~$C_6$. A simple calculation shows that
the projective closure of~$C_6$ in~$\PP^2$ is singular at~$(0,0)$
and at the point at infinity, and that each of these singular points
resolves to three points on the smooth model.
(See Proposition~\ref{proposition:GKzformula} for details.)
We let~$C_1$ be the elliptic curve
\[
  C_1 : z(1-z) = x^3,
\]
and we consider the natural degree~$6$ map
\[
  \psi: C_6 \longrightarrow C_1,\qquad
  (x,z) \longmapsto (x,z^6).
\]
The map~$\psi$ is ramified only at~$(0,0)$ and~$\infty$,
the sets $\psi^{-1}(0,0)$ and $\psi^{-1}(\infty)$ each consist of three
points,  and each of these points has ramification index~$2$.
Applying the Riemann--Hurwitz genus formula to~$\psi$ gives
\[
  2g(C_6)-2 = 6\bigl(2g(C_1)-2\bigr) 
  + \sum_{P\in C_1} \bigl(e_P(\psi)-1\bigr)
  = 6(2-2) + 6(2-1)
  = 6.
\]
Hence $g(C_6)=4$.
\par\noindent(b)\enspace
It is an exercise to verify that the given maps are well-defined, but
we briefly comment on their origin.  The automorphism group
of the curve~$C_6^{(\g,\d)}$ is fairly large, since 
\[
  \bfmu_3\times\bfmu_6\subset \Aut(C_6^{(\g,\d)}), \qquad
  [\z,\xi](x,z) = (\z x,\xi z).
\]
Taking quotients of~$C_6^{(\g,\d)}$ by various subgroups
of~\text{$\bfmu_3\times\bfmu_6$} gives maps to curves of lower genus,
which in turn give the four maps described in~(b).
\par\noindent(c)\enspace
From general principles, the maps in~(b) induce isogenies $E^{(\k)}\to
J_6^{(\g,\d)}$ for the given values of~$\k$.  There are various ways
to see that these isogenies are independent.  For example, one can use
the fact that the four~$E^{(\k)}$ are non-isogenous over~$\CC(\g,\d)$,
treating~$\g$ and~$\d$ as indeterminates.  
Or, at least in characteristic~$0$, one can take~$\g=\d=1$,
untwist to get four maps~$C_6^{(1,1)}\to E^{(1)}$ defined over~$\Qbar$,
and use the action of~$\Gal(\Qbar/\QQ)$ on the maps to show
that they are independent. (See Appendix~\ref{appendix:indepofmaps}.)
Or, for a purely geometric proof, one can use intersection theory and
the fact that the pairing
\begin{gather*}
  \langle\,\cdot\,,\,\cdot\,\rangle
  : \operatorname{Map}(C_6^{(1,1)},E^{(1)})/E^{(1)}\to\ZZ,\\
  \langle\f,\psi\rangle = \deg(\f+\psi)-\deg\f-\deg\psi,
\end{gather*}
is a positive definite quadratic form. (The~$E^{(1)}$ in the denominator
is shorthand for the right action of the group of translations.)
\par\noindent(d)\enspace
It is a standard fact that~$H^1_\et$ of a curve and its Jacobian are
isomorphic. This gives the first isomorphism, and the second follows
from~(c) and the fact that an isogeny between abelian varieties
induces an isomorphism of their \'etale cohomologies.
\end{proof}

\begin{proposition}
\label{proposition:GKzformula}
Let $\gK$ be a prime ideal in~$\Ocal_K$ such
that~$\bfmu_6\subset\Ocal_K/\gK$, i.e.,
$\Norm_{K/\QQ}(\gK)\equiv1\pmod6$.  Let $\z\in\bfmu_6$ and
$\xi\in\bfmu_3$, choose elements~$\g,\d\in\Ocal_K$ satisfying
$\SS{\g}{\gK}=\z$ and $\CS{\d}{\gK}=\xi$, and let~$C_6^{(\d,\g)}$ be
the smooth projective curve from
Proposition~$\ref{proposition:ECDcurves}$ given by the affine equation
\[
  C_6^{(\g,\d)} : \g z^6(1-\g z^6) = \d x^3.
\]
Then
\[
  \#M_\gK^{[1]}(\z,\xi) = \frac{1}{18}\left(
    \#C_6^{(\g,\d)}\left(\frac{\Ocal_K}{\gK}\right) - e(\z,\xi) 
  \right),
\]
where the error term $e(\z,\xi)$ is given by the formula
\[
  e(\z,\xi) = 
  \left[\begin{tabular}{@{}cl@{}}
     $6$&if $\z=1$\\
     $0$&if $\z\ne1$\\
     \end{tabular}\right]
  +\left[\begin{tabular}{@{}cl@{}}
     $3$&if $\z^2=\xi$\\
     $0$&if $\z^2\ne\xi$\\
     \end{tabular}\right]
  +\left[\begin{tabular}{@{}cl@{}}
     $3$&if $\z^4=\xi$\\
     $0$&if $\z^4\ne\xi$\\
     \end{tabular}\right].
\]
\end{proposition}
\begin{proof}
Our choice of~$\g$ and~$\d$ imply that for any~$\l\in\Ocal_K$,
\begin{align*}
  \SS{\l}{\gK} = \z 
  &\quad\Longleftrightarrow\quad
  \g^{-1}\l \equiv \text{non-zero sixth power}\pmod\gK, \\  
  \CS{\l(1-\l)}{\gK} = \xi
  &\quad\Longleftrightarrow\quad
  \d^{-1}\l(1-\l) \equiv \text{non-zero cube}\pmod\gK.
\end{align*}
We thus get a natural map
\begin{equation}
  \label{eqn:C6toMkzxi}
  \begin{array}{rcl}
    \left\{(x,z) \in C_6^{(\g,\d)}\left(\dfrac{\Ocal_K}{\gK}\right)
       : x \ne 0,\infty \right\}
    &\longrightarrow&
    M_\gK(\z,\xi), \\
    (x,z) &\longmapsto& \g z^6.
  \end{array}
\end{equation}
We claim that the map~\eqref{eqn:C6toMkzxi} is exactly $18$-to-$1$.
To see this, let $\l\in M_\gK(\z,\xi)$. Then $\l\equiv\g
v^6\pmod{\gK}$ and $\l(1-\l)\equiv\d u^3\pmod{\gK}$ for some
$u,v\in(\Ocal_K/\gK)^*$, so~$\l$ is the image of the point $(u,v)\in
C_6^{(\g,\d)}(\Ocal_K/\gK)$. Further, for a given value of~$\l$, there
are six choices for~$v$ and three choices for~$u$. (Note that
$\Ocal_K/\gK$ contains~$\bfmu_6$.) Hence
\[
  \#M_\gK(\z,\xi) = \frac{1}{18}
    \#\left(C_6^{(\g,\d)}\left(\frac{\Ocal_K}{\gK}\right)\setminus
       \{x=0~\text{or}~\infty\}\right).
\]
It remains to count the number of $\Ocal_K/\gK$-rational points
with $x=0$ or~$\infty$ on a smooth model of~$C_6^{(\g,\d)}$.
\par
To ease notation, we let $C=C_6^{(\g,\d)}$, and we let~$C'$ be the curve
\begin{equation}
  \label{eqn:Cprmg1gz6}
  C' : \g(1-\g z^6) = \d x^3.
\end{equation}
The birational map
\[
  C \longrightarrow C',\qquad
  (x,z) \longmapsto (xz^{-2},z),
\]
is a bijection on the set of points
\[
  C \setminus \{x=0~\text{or}~\infty\}
  \quad\stackrel\sim\longleftrightarrow\quad
  C' \setminus \{x=0~\text{or}~\infty\} \cup \{z=0\},
\]
and the affine piece of~$C'$ defined by equation~\eqref{eqn:Cprmg1gz6}
is smooth, so the points with~$x=0$ on~$C$ become the points
with~$x=0$ or~$z=0$ on~$C'$. (More precisely, we will see that the
singular point $(0,0)\in C$ is blown up to three points on~$C'$, while
there are six smooth points of the form $(0,\g^{-1/6})$ on both~$C$
and~$C'$.)  The points on~$C'$ with~$x=0$ or $z=0$ are characterized
by
\[
  (0,z)\in C' \Longleftrightarrow z^6=\g^{-1}
  \qquad\text{and}\qquad
  (x,0)\in C' \Longleftrightarrow x^3=\g\d^{-1}.
\]
Thus there are points of the form~$(0,z)$ if and only
if~$\SS{\g}{\gK}=1$, and there are points of the form~$(0,x)$ if and
only if~$\CS{\g\d^{-1}}{\gK}=1$.  Using the values $\SS{\g}{\gK}=\z$
and
$\CS{\g\d^{-1}}{\gK}=\SS{\g}{\gK}^2\CS{\d^{-1}}{\gK}=\z^2\xi^{-1}$, we
find that
\begin{align*}
 \#\left\{(0,z)\in C_6^{(\g,\d)}\left(\frac{\Ocal_K}{\gK}\right) \right\}
  &=\begin{cases}
      6&\text{if $\z=1$,}\\
      0&\text{if $\z\ne1$,}\\
  \end{cases} \\
 \#\left\{(x,0)\in C_6^{(\g,\d)}\left(\frac{\Ocal_K}{\gK}\right) \right\}
  &=\begin{cases}
      3&\text{if $\z^2=\xi$,}\\
      0&\text{if $\z^2\ne\xi$.}\\
  \end{cases} 
\end{align*}
\par
It remains to count the points at infinity on~$C'$. Homogenizing the
equation for~$C'$ gives the curve
$\g y^6 - \g^2 z^6=\d x^3y^3$. The unique (singular) point
at infinity is $[x,y,z]=[1,0,0]$, so dehomogenizing by setting $x=1$ gives
the curve
\[
  \g y^6 - \g^2 z^6 = \d y^3.
\]
We blow up the singular point $(0,0)$ by setting $y=z^2u$. (This
corresponds to blowing up twice. One can check that the other
coordinate charts do not yield any additional points.)  The resulting
curve has affine equation
\[
  \g z^6u^6 - \g^2 = \d u^3.
\]
This affine curve is smooth, and the points that map to the point
at infinity on~$C'$ are the points with $z=0$ and $u^3=-\g^2\d^{-1}$.
Using $\CS{\g^2\d^{-1}}{\gK}=\SS{\g}{\gK}^4\CS{\d^{-1}}{\gK}$, we
see that
\[
 \#\left\{\text{points at infinity on }
    C_6^{(\g,\d)}\left(\frac{\Ocal_K}{\gK}\right) \right\}
  =\begin{cases}
      3&\text{if $\z^4=\xi$,}\\
      0&\text{if $\z^4\ne\xi$.}\\
  \end{cases} \\
\]  
This completes the proof of the proposition.
\end{proof}

The next step is to count the number of points on~$C_6^{(\g,\d)}$ defined
over a finite field. This is done using the decomposition of~$J_6^{(\g,\d)}$
into a product of elliptic curves.

\begin{proposition}
\label{proposition:numC6pts}
With notation as in the statement of
Proposition~$\ref{proposition:GKzformula}$,
choose an element~$\pi\in\Ocal_K$ satisfying $\gK=\pi\Ocal_K$
and $\pi\equiv2\pmod3$. Further let $\e=\CS{2}{\gK}$. Then
\begin{multline*}
  \#C_6^{(\g,\d)}\left(\frac{\Ocal_K}{\gK}\right)
  = \Norm_{K/\QQ}\gK + 1 
  + \Trace_{K/\QQ}(\xi\bar\pi)
  + \Trace_{K/\QQ}(\e^2\z^3\xi^2\bar\pi) \\
  + \Trace_{K/\QQ}(\e\z^5\xi\bar\pi)
  + (-1)^{\frac12(\Norm_{K/\QQ}\gK-1)}\Trace_{K/\QQ}(\e\z\xi\bar\pi).
\end{multline*}
If $\gK$ is an inert prime, say $\gK=k\Ocal_K$ with $k\in\ZZ$
satisfying $k\equiv2\pmod3$, and if we take $\d=1$,
then the formula simplifies to
\[
  \#C_6^{(\g,1)}\left(\frac{\Ocal_K}{\gK}\right)
  = \begin{cases}
    k^2 + 1 + 8k &\text{if $\z=1$,} \\
    k^2 + 1 - 4k &\text{if $\z=-1$,} \\
    k^2 + 1 + 2k &\text{if $\z\ne\pm1$.} \\
  \end{cases}
\]
\end{proposition}
\begin{proof}
To ease notation, let $\FF_\gK=\Ocal_K/\gK$, so
$\Norm_{K/\QQ}\gK=\#\FF_\gK$.  Further let $F_\gK$ be the
$(\Norm_{K/\QQ}\gK)^{\text{th}}$-power Frobenius map on~$\bar\FF_\gK$.
Then the number of points in $C_6^{(\g,\d)}(\FF_\gK)$ is given by the
trace formula~\cite[C.4.2]{Hartshorne},
\begin{equation}
  \label{eqn:DfFk2k21x}
  \#C_6^{(\g,\d)}(\FF_\gK)
  = \Norm_{K/\QQ}\gK + 1 - \Trace\Bigl(F_\gK
       \bigm| H^1_\et({C_{6/\FF_\gK}^{(\g,\d)}},\QQ_\ell)\Bigr).
\end{equation}
We  compute the trace using
Proposition~\ref{proposition:ECDcurves}, which splits the
representation for~$C_6^{(\g,\d)}$ into a product of representations on
elliptic curves with zero $j$-invariant. Thus
\begin{multline}
  \label{eqn:trH1D1E16E4x}
  \Trace\Bigl(F_\gK \bigm| H^1_\et({C_{6/\FF_\gK}^{(\g,\d)}},\QQ_\ell)\Bigr)\\
  =
  \Trace\Bigl(F_\gK \bigm| H^1_\et({E_{/\FF_\gK}^{(16\d^2)}},\QQ_\ell)\Bigr)
  + \Trace\Bigl(F_\gK \bigm| H^1_\et({E_{/\FF_\gK}^{(4\g^3\d^4)}},\QQ_\ell)\Bigr) \\
  + \Trace\Bigl(F_\gK \bigm| H^1_\et({E_{/\FF_\gK}^{(\g^5\d^2)}},\QQ_\ell)\Bigr)
  + \Trace\Bigl(F_\gK \bigm| H^1_\et({E_{/\FF_\gK}^{(-\g\d^2)}},\QQ_\ell)\Bigr).
\end{multline}
\par
We now apply~\cite[Chapter~18, Section~3, Theorem~4]{IrelandRosen},
which gives a formula for the trace in terms of residue symbols.
Writing $\gK=\pi\Ocal_K$ with $\pi\equiv2\pmod3$, we find that
\begin{align}
  \label{eqn:GK3x}
  \Trace&\Bigl(F_\gK \bigm| H^1_\et({C_{6/\FF_\gK}^{(\g,\d)}},\QQ_\ell)\Bigr)\notag\\*
  &=  -\SS{2^6\d^2}{\gK}^{-1}\pi - \SS{2^6\d^2}{\gK}\bar\pi
    - \SS{2^4\g^3\d^4}{\gK}^{-1}\pi - \SS{2^4\g^3\d^4}{\gK}\bar\pi \notag\\*
  &\quad{}  - \SS{2^2\g^5\d^2}{\gK}^{-1}\pi - \SS{2^2\g^5\d^2}{\gK}\bar\pi 
    - \SS{-2^2\g\d^2}{\gK}^{-1}\pi - \SS{-2^2\g\d^2}{\gK}\bar\pi \notag\\
  &= -\xi^{-1}\pi - \xi\bar\pi
   - \CS{2}{\gK}^{-2}\z^{-3}\xi^{-2}\pi - \CS{2}{\gK}^{2}\z^{3}\xi^{2}\bar\pi
       \notag\\*
  &\quad{} - \CS{2}{\gK}^{-1}\z^{-5}\xi^{-1}\pi - \CS{2}{\gK}\z^5\xi\bar\pi
       \notag\\*
  &\quad{} - \LS{-1}{\gK}_2\CS{2}{\gK}^{-1}\z^{-1}\xi^{-1}\pi
       - \LS{-1}{\gK}_2\CS{2}{\gK}\z\xi\bar\pi.
\end{align}
Noting that $\LS{-1}{\gK}_2=(-1)^{(\Norm_{K/\QQ}\gK-1)/2}$, we
combine~\eqref{eqn:DfFk2k21x}
and~\eqref{eqn:GK3x} to obtain the desired result.
\par
In the case that $\gK=k\Ocal_K$ is an inert prime, we have
$(-1)^{\frac12(k^2-1)}=1$ since~$k$ is odd.  Further, both~$2$ and~$k$
are primary, so cubic reciprocity gives $\CS{2}{\gK} = \CS{2}{k} =
\CS{k}{2} = 1$. Further taking~$\d=1$ implies that~$\xi=1$, so the
formula for~$\#C_6^{(\g,1)}(\Ocal_K/\gK)$ becomes
\[
  k^2 + 1 
  + \bigl(  \Trace_{K/\QQ}(1)
  + \Trace_{K/\QQ}(\z^3) 
  + \Trace_{K/\QQ}(\z^5)
  + \Trace_{K/\QQ}(\z) \bigr)k.
\]
Taking the six possible values $\z\in\bfmu_6$  yields the stated formula.
\end{proof}

\begin{proposition}
\label{propostion:Mk1Stable}
Let $k\ge5$ be a rational prime.  The following table gives the values
of~$\#M_k^{[1]}(S,1)$ for various subsets $S\subset\bfmu_6$, divided into
cases according to whether~$k$ is split or inert
in~$K=\QQ(\sqrt{-3}\,)$, cf.\ Proposition~$\ref{propostion:MkStable}$.
\begin{center}
\renewcommand{\arraystretch}{1.25}
\begin{tabular}{|c|c||c|c|} \cline{3-4}
  \multicolumn{2}{c||}{}
  &$k\equiv1\pmod3$&$k\equiv2\pmod3$\\ \hline\hline
  \textup{(a)} & $\#M_{k}^{[1]}\bigl(\{\o,\o^5\},1\bigr)$
     & $\frac19(k-1)^2$ & $\frac19(k+1)^2$ \\ \hline
  \textup{(b)} & $\#M_{k}^{[1]}\bigl(\{\o,\o^3,\o^5\},1\bigr)$
     & $\frac16(k-1)(k-3)$ & $\frac16(k^2-1)$ \\ \hline
  \textup{(c)} & $\#M_{k}^{[1]}\bigl(\{\o,\o^2,\o^4,\o^5\},1\bigr)$
     & $\frac{2}{9}(k-1)^2$ & $\frac29(k+1)^2$ \\ \hline
  \textup{(d)} & $\#M_{k}^{[1]}(\bfmu_6,1)$
     & $\frac13(k^2-2k+4)$ & $\frac13(k^2+2k-2)$ \\ \hline
\end{tabular}
\end{center}
\end{proposition}
\begin{proof}
We begin with the case that $k\equiv2\pmod3$, so $\gK=k\Ocal_K$ is a
prime ideal with $\Norm_{K/\QQ}\gK=k^2$.  We let
$\o=\frac12(1+\sqrt{-3})$ be the usual sixth root of unity, and we
choose some~$\g\in\Ocal_K$ satisfying
\[
  \SS{\g}{\gK} = \o.
\]
Then for any $0\le i\le 5$ we have
\begin{align*}
  18 \#M_k^{[1]}(\o^i,1) 
  &= \#C_6^{(\g^i,1)}(\FF_\gK) 
       -\left[\begin{tabular}{@{}rl@{}} 
          12& if $i=0$\\
          6& if $i=3$\\
          0& otherwise\\
         \end{tabular}\right]
      \\
  &\omit\hfill from Proposition \ref{proposition:GKzformula}
     with $\z=\o^i$ and $\xi=1$, \\
  &=  \left[\begin{tabular}{@{}rl@{}} 
         $k^2 + 1 + 8k$ &\text{if $i=0$} \\
         $k^2 + 1 - 4k$ &\text{if $i=3$} \\
         $k^2 + 1 + 2k$ &\text{otherwise} \\
       \end{tabular}\right]
       -\left[\begin{tabular}{@{}rl@{}} 
          12& if $i=0$\\
          6& if $i=3$\\
          0& otherwise\\
         \end{tabular}\right]
      \hspace{2em}\\
    &\omit\hfill from Proposition \ref{proposition:numC6pts}
     with $\z=\o^i$ and $\xi=1$, \\
  &= \begin{cases}
     k^2 + 8k - 11 &\text{if $i=0$,} \\
     k^2 - 4k - 5  &\text{if $i=3$,} \\
     k^2 + 2k + 1  &\text{otherwise.}
    \end{cases}
\end{align*}
It is now easy to compute~$\#M_k^{[1]}(S,1)=\sum_{\z\in S}\#M_k^{[1]}(\z,1)$
for the four cases of the proposition. For example,
\begin{align*}
  \#M_{k}^{[1]}(\bfmu_6,1)
  &= \frac{1}{18}\bigl((k^2+8k-11)+(k^2-4k-5) + 4(k^2+2k+1)\bigr)\\*
    &   = \frac{1}{3}(k^2+2k-2).
\end{align*}
\par
Next we consider the case that $k\equiv1\pmod3$,
so~$k\Ocal_K=\gK\bar\gK$ splits.  The definition of the residue symbol
says that
\begin{align*}
  \SS{\l}{k\Ocal_K}&=\SS{\l}{\gK}\SS{\l}{\bar\gK},\\
  \CS{\l(1-\l)}{k\Ocal_K}&=\CS{\l(1-\l)}{\gK}\CS{\l(1-\l)}{\bar\gK},
\end{align*}
so using the Chinese remainder theorem
\[
  \Ocal_K/k\Ocal_K = \Ocal_K/\gK\Ocal_K \times \Ocal_K/\bar\gK\Ocal_K,
\]
a quantity such as $M_k^{[1]}(\z,\xi)$  breaks up into a sum of products,
\[
  M_k^{[1]}(\z,\xi) = \sum_{u=0}^5 \sum_{v=0}^2 
          M_\gK^{[1]}(\o^u,\o^{2v})M_{\bar\gK}^{[1]}(\z\o^{-u},\xi\o^{-2v}).
\]
Hence for $0\le i\le 5$ we have
\begin{align}
  \label{eqn:splitct1}
  M_k^{[1]}(\o^i,1) &= \sum_{u=0}^5 \sum_{v=0}^2 
          M_\gK^{[1]}(\o^u,\o^{2v})M_{\bar\gK}^{[1]}(\o^{i-u},\o^{-2v}) \notag\\
  &= \sum_{u=0}^5 \sum_{v=0}^2 
          M_\gK^{[1]}(\o^u,\o^{2v})M_{\gK}^{[1]}(\o^{u-i},\o^{2v}).
\end{align}
(For the second equality we've used the identity
$M_{\gK}^{[1]}(\z,\xi)=M_{\bar\gK}^{[1]}(\bar\z,\bar\xi)$.)  We
choose~$\g$ and~$\d$ to satisfy
\[
  \SS{\g}{\gK} = \o
  \qquad\text{and}\qquad
  \CS{\d}{\gK} = \o^2.
\]
Then Proposition~\ref{proposition:GKzformula} gives us the formula
\begin{equation}
  \label{eqn:splitct2}
  18 M_\gK^{[1]}(\o^u,\o^{2v})
  = \#C_6^{(\g^u,\d^v)}(\FF_\gK) - e(\o^u,\o^{2v}),
\end{equation}
where
\begin{align}
  \label{eqn:splitct3}
  e(\o^u,\o^{2v}) = 
  \left[\begin{tabular}{@{}cl@{}}
     $6$&if $u\equiv 0\pmod6$\\
     $0$&otherwise\\
     \end{tabular}\right]
  & +\left[\begin{tabular}{@{}cl@{}}
     $3$&if $u\equiv v\pmod3$\\
     $0$&otherwise\\
     \end{tabular}\right]
  \notag\\
  & + \left[\begin{tabular}{@{}cl@{}}
     $3$&if $2u\equiv v\pmod3$\\
     $0$&otherwise\\
     \end{tabular}\right].
\end{align}
Further, Proposition~\ref{proposition:numC6pts} gives us the number of
points on the curve,
\begin{multline}
  \label{eqn:splitct4}
  \#C_6^{(\g^u,\d^v)}(\FF_\gK)
  = k + 1 + \Trace_{K/\QQ}(\o^{2v}\bar\pi) + \Trace_{K/\QQ}(\e^2\o^{3u+4v}\bar\pi)\\
   + \Trace_{K/\QQ}(\e\o^{5u+2v}\bar\pi) 
           + (-1)^{\frac12(k-1)}\Trace_{K/\QQ}(\e\o^{u+2v}\bar\pi),
\end{multline}
where $\e=\LS{2}{\gK}_3$.
\par
Combining~\eqref{eqn:splitct1}, \eqref{eqn:splitct2},
\eqref{eqn:splitct3}, and~\eqref{eqn:splitct4} gives an explicit,
albeit quite complicated, formula for~$M_k^{[1]}(\o^i,1)$.  In
principle, this formula could be computed by hand, but we are content
to evaluate it using PARI~\cite{PARI}, which yields the following
values:
\begin{align*}
  18\#M_k^{[1]}(\o^i,1) =
    \begin{cases}
      k^2 + 4 k + 13 & \text{if $i=0$,} \\
      k^2 - 8 k + 7  & \text{if $i=3$,} \\
      k^2 - 2 k + 1  & \text{otherwise.} \\
    \end{cases}
\end{align*}
(See Remark~\ref{remark:Mksplitcomputation} for further information
about this computation.) It is now a simple matter to compute the
value of~$\#M_k^{[1]}(S,1)$ for the four cases.  For example,
\begin{align*}
  \#M_k^{[1]}(\bfmu_6,1) &=
    \frac{1}{18}\bigl((k^2 + 4 k + 13) + (k^2-8k+7)
      + 4(k^2-2k+1)\bigr) \\*
    &   = \frac{1}{3}(k^2-2k+4).
\end{align*}
This completes the proof of Proposition~\ref{propostion:Mk1Stable}.
\end{proof}

\begin{remark}
\label{remark:Mksplitcomputation}
We used PARI~\cite{PARI} to compute $M_k^{[1]}(\o^i,1)$ by evaluating
formulas~\eqref{eqn:splitct1}, \eqref{eqn:splitct2},
\eqref{eqn:splitct3}, and~\eqref{eqn:splitct4}, where we
treated~$k$,~$\pi$, and~$\e$ as indeterminates, and we formally set
$\bar\pi=k/\pi$ and $\bar\e=1/\e$.  
(The PARI code used for the computation is given in
Appendix~\ref{appendix:pariprogram}.)
The value of $M_k^{[1]}(\o^i,1)$
turns out to be a quadratic polynomial in~$k$ that is independent of
$k\bmod4$. We do not have an \emph{a priori} explanation for why this should
be the case. In order to illustrate the delicacy of the argument, we
suppose for a moment that the isogeny decomposition of the Jacobian of
$C_6^{(\g,\d)}$ in Proposition~\ref{proposition:ECDcurves} looks like
\[
  E^{(16\d^2)}\times E^{(4\g^4\d^4)}\times E^{(\g^5\d^2)}\times E^{(-\g\d^2)}
  \longrightarrow  \Jac(C_6^{(\g,\d)}) .
\]
(All that we have done is change the second elliptic factor from
$E^{(4\g^3\d^4)}$ to $E^{(4\d^4)}$.)  This would have the effect
in formula~\eqref{eqn:splitct4} of changing the second trace term from
$\Trace_{K/\QQ}(\e^2\o^{3u+4v}\bar\pi)$ to
$\Trace_{K/\QQ}(\e^2\o^{4v}\bar\pi)$. But with this small
modification, there is less cancelation in the computation
of~$M_k^{[1]}(\o^i,1)$, so for example
$\#M_k^{[1]}\bigl(\{\o,\o^5\},1\bigr)$ would equal
\[
  \frac{1}{9}\left(
    k^2 + 2k + 1 + 2\Trace\left(\CS{2}{\gK}^{2} \bar\pi^2\right)\right).
\]
Thus $\#M_k^{[1]}\bigl(\{\o,\o^5\},1\bigr)$ would depend on
both~$\CS{2}{\gK}$ and on the factorization of~$k$ in~$\Ocal_K$.
\end{remark}

\begin{remark}
Many of the cases of Proposition~\ref{propostion:Mk1Stable} can be
obtained somewhat more easily by working on elliptic curves $z(1-z)=\d
x^3$ or genus two curves $\g z^2(1-\g z^2)=\d x^3$. However, some
cases require the curves $\g z^6(1-\g z^6)=\d x^3$ of genus four, so
for unity of exposition and to save space, we have derived all cases
using these latter curves.
\end{remark}

Combining Conjecture~\ref{conjecture:N1kN1M1kM1} with the computations
in Propositions~\ref{propostion:MkStable}
and~\ref{propostion:Mk1Stable} yields precise formulas for the
conjectural density of Type~1 primes on~$y^2=x^3+k$ when~$k$ is prime.

\begin{conjecture}
\label{conjecture:N1density}
Let $k\ge5$ be a rational prime. Then
\[
  \lim_{X\to\infty} \frac{\Ncal^{[1]}_k(X)}{\Ncal_k(X)}
  = \frac13 + R(k),
\]
where~$R(k)$ depends on $k\pmod{36}$ and is given by the following 
table\textup:
\[
  \renewcommand{\arraystretch}{2.5}
  \begin{array}{|c||c|c|} \hline
  & k\bmod36 
  & R(k) \\ \hline\hline
  \textup{(a,c)} & 1,\;19 & \dfrac{2}{3(k-3)} \\ \hline
  \textup{(b)} & 13,\;25 & 0  \\ \hline
  \textup{(d)} & 7,\;31 & \dfrac{2k}{3(k-2)^2}  \\ \hline
  \end{array}
  \qquad
  \begin{array}{|c||c|c|} \hline
  & k\bmod36 
  & R(k) \\ \hline\hline
  \textup{(a,c)}  & 17,\;35 & \dfrac{2}{3(k-1)} \\ \hline
  \textup{(b)} & 5,\;29 & 0 \\ \hline
  \textup{(d)} & 11,\;23 & \dfrac{2k}{3(k^2-2)} \\ \hline
  \end{array}
\]
In particular, $R(k)=O(1/k)$.
\end{conjecture}

We do not have an intrinsic explanation for why~$R(k)$ is the same in
cases~(a) and~(c), nor do we know why~$R(k)=0$ in case~(b).

\section{Amicable pairs for elliptic curves --- Experiments}
\label{section:amicableCMexper}
In this section we present the results of experiments that test the
reasonableness of our conjectures. We begin with
Conjecture~\ref{conj:amicablepair}, which deals with the case of CM
curves having nonzero $j$-invariant.
\par
We computed the number~$\Qcal_E(X)$ of amicable pairs up to~$X$ for
elliptic curves with~CM by the imaginary quadratic order of
discriminant~$-D$ and conductor~$f$.  Theorem~\ref{thm:amicableCMconj}
says that it suffices to consider $D\equiv3\pmod4$. Further, the
assumption that~$E$ is defined over~$\QQ$ means that~$\Ocal$ has class
number one, so the classification of imaginary quadratic fields of
class number one combined with an elementary formula for the class
number of an order~\cite[Exercise~4.12]{Shimura} imply that the only
possibilities for~$D$ are~$D\in\{3,7,11,19,43,67,163\}$, and the
possible values of~$f$ are given by $f\in\{1,2,3\}$ if $D=3$,
$f\in\{1,2\}$ if $D=7$, and $f=1$ in all other cases.
See~\cite[A~\S3]{ATAEC} for a Weierstrass equations for each~CM type.
\par
We ignore for the moment the case $(D,f)=(3,1)$.  As noted in the
proof of Theorem~\ref{thm:amicableCMconj}, the curves with $(D,f)$
equal to $(3,2)$,~$(7,1)$, and~$(7,2)$ have nontrivial $2$-torsion, so
neither they nor any of their (necessarily quadratic) twists have
amicable pairs. The curve with $(D,f)=(3,3)$ listed
in~\cite[A~\S3]{ATAEC} has nontrivial $3$-torsion, but it has
quadratic twists with trivial torsion, so is a candidate to have
amicable pairs.  Table~\ref{tableCM1} lists the number~$\Qcal_E(X)$ of
amicable pairs up to the given bound and the ratio of~$\Qcal_E(X)$ to
the number~$\Ncal_E(X)$ of primes~$p$ such that~$\#\Etilde_p(\FF_p)$
is prime.  For this table we used the following Weierstrass
equations.\footnote{Calculations on quadratic twists of the listed
  curves yielded virtually identical results.}
\begin{align*}
  (D,f) &=( 3,3) & y^2 &= x^3  -120x + 506, \\
  (D,f) &=( 11,1) & y^2 +y&= x^3 - x^2 -7x +10, \\
  (D,f) &=( 19,1) & y^2+y &= x^3 -38x+ 90, \\
  (D,f) &=( 43,1) & y^2+y &= x^3 -860x+ 9707, \\
  (D,f) &=( 67,1) & y^2+y &= x^3 -7370x+ 243528, \\
  (D,f) &=( 163,1) & y^2+y &= x^3 -2174420x+ 1234136692.
\end{align*}
The results in Table~\ref{tableCM1} are consistant with
Conjecture~\ref{conj:CMamicable}, which predicts that the
ratio~$\Qcal_E(X)/\Ncal_E(X)$ should approach~$\frac14$.

\begin{table}
\begin{center}
\renewcommand{\arraystretch}{1.25}
\begin{tabular}{|r||*{10}{r|}} \hline
  $(D,f)$ & (3,3) & (11,1) & (19,1) & (43,1) & (67,1) & (163,1) \\\hline\hline
  $\Qcal_E(10^{5})$  & 124 & 48 & 103 & 205 & 245 & 395 \\ \hline
  $\Qcal_E(10^{5})/\Ncal_E(10^{5})$  & 0.251 & 0.238 & 0.248 & 
       0.260 & 0.238 & 0.246 \\ 
  \hline\hline
  $\Qcal_E(10^{6})$  & 804 & 303 & 709 & 1330 & 1671 & 2709 \\ \hline
  $\Qcal_E(10^{6})/\Ncal_E(10^{6})$  & 0.250 & 0.247 & 0.253 & 
       0.255 & 0.245 & 0.247 \\ 
  \hline\hline
  $\Qcal_E(10^{7})$  & 5581 & 2267 & 5026 & 9353 & 12190 & 19691 \\ \hline
  $\Qcal_E(10^{7})/\Ncal_E(10^{7})$  & 0.249 & 0.251 & 0.250 & 
       0.251 & 0.250 & 0.252 \\ \hline
\end{tabular}
\end{center}
\caption{$\Qcal_E(X)$ and $\Qcal_E(X)/\Ncal_E(X)$
for elliptic curves with CM by $\QQ(\sqrt{-D})$}
\label{tableCM1}
\end{table}

We next considered the curves $y^2=x^3+k$ with $j(E)=0$.
Table~\ref{tableCM5}, which is included for historical reasons, was
our first intimation that the limiting value
of~$\Qcal_E(X)/\Ncal_E(X)$  behaves differently for different
values of~$k$, with no obvious pattern for $2\le k\le 10$. (Note that
we do not list values of~$k$ that are squares or cubes, since in those
cases~$E(\QQ)_\tors$ is nontrivial, so there are no amicable pairs.)

\begin{table}
\begin{center}
\renewcommand{\arraystretch}{1.25}
\begin{tabular}{|r||*{10}{r|}} \hline
  $k$  & 2 & 3 & 5 & 6 & 7 & 10 \\ \hline\hline
  $X = 10^{5}$  & 0.251 & 0.122 & 0.081 & 0.134 & 0.139 & 0.125 \\ \hline
  $X = 10^{6}$  & 0.250 & 0.139 & 0.083 & 0.142 & 0.133 & 0.107 \\ \hline
  $X = 10^{7}$  & 0.249 & 0.139 & 0.082 & 0.139 & 0.129 & 0.107 \\ \hline
\end{tabular}
\end{center}
\caption{$\Qcal_E(X)/\Ncal_E(X)$ for elliptic curves $y^2=x^3+k$}
\label{tableCM5}
\end{table}

We recall the notation~$\Ncal_k^{[1]}$ for the set of Type~1 primes
for the curve~$y^2=x^3+k$; see Section~\ref{section:amicableCMj0} for
the precise definition.  Conjecture~\ref{conj:QoverNeq12} predicts
that $\Qcal_k(X)\sim\frac14\Ncal_k^{[1]}(X)$, and in the case that~$k$
is prime, Conjecture~\ref{conjecture:N1density} says that
\[
  \Ncal_k^{[1]}(X) \sim \left(\frac13+R(k)\right)  \Ncal_k(X),
\]
where~$R(k)$ is given by an explicit formula that depends on~$k$
modulo~$36$. We tested these two conjectures by computing
$\Qcal_k(X)$, $\Ncal_k^{[1]}(X)$, and $\Ncal_k(X)$ for~$X=10^8$. The
results are listed in Table~\ref{tableCM9}. Column~5 provides
convincing evidence for Conjecture~\ref{conj:QoverNeq12},  and the
final two columns show that Conjecture~\ref{conjecture:N1density} is
in good agreement with experiment in all eight cases.  (The
notation~$(x.n)$ after each value of~$k$ indicates the case
$x=\text{(a),\dots,(d)}$ and the congruence class~$k\equiv n\pmod3$
from Conjecture~\ref{conjecture:N1density}.)

\begin{table}[ht]
\tiny
\begin{center}
\renewcommand{\arraystretch}{1.25}
\begin{tabular}{|r||*{10}{c|}} 
  \cline{6-7}
  \multicolumn{5}{c}{} & \multicolumn{2}{|c|}{Density of Type 1 primes} \\
  \multicolumn{5}{c}{} & \multicolumn{2}{|c|}{$\Ncal_k^{[1]}(X)/\Ncal_k(X)$} \\
  \hline
  \hfill\hfill$k$\hfill\hfill & $\Qcal_k(X)$ & $\Ncal_k^{[1]}(X)$ & $\Ncal_k(X)$ 
   & $\Qcal/\Ncal^{[1]}$ & experiment & conjecture \\
  \hline\hline
\text{5 (b.2)}
 &29340&58594&175703&0.251&0.3335&$\frac{1}{3} = 0.3333$ \\ \hline
\text{7 (d.1)}
 &43992&87825&168743&0.251&0.5205&$\frac{13}{25} = 0.5200$ \\ \hline
\text{11 (d.2)} 
 &33721&66698&169062&0.253&0.3945&$\frac{47}{119} = 0.3950$ \\ \hline
\text{13 (b.1)}
 &28036&55766&167333&0.252&0.3333&$\frac{1}{3} = 0.3333$ \\ \hline
\text{17 (a.2)}
 &32008&63810&169226&0.251&0.3771&$\frac{3}{8} = 0.3750$ \\ \hline
\text{19 (c.1)} 
 &31729&63066&168196&0.252&0.3750&$\frac{3}{8} = 0.3750$ \\ \hline
\text{23 (d.2)}
 &30480&61210&168512&0.249&0.3632&$\frac{191}{527} = 0.3624$ \\ \hline
\text{29 (b.2)}
 &28085&56286&168642&0.249&0.3338&$\frac{1}{3} = 0.3333$ \\ \hline
\text{31 (d.1)}
 &30301&60349&168344&0.251&0.3585&$\frac{301}{841} = 0.3579$ \\ \hline
\text{37 (a.1)}
 &29728&59430&168471&0.250&0.3528&$\frac{6}{17} = 0.3529$ \\ \hline
\text{41 (b.2)}
 &28050&56381&168567&0.249&0.3345&$\frac{1}{3} = 0.3333$ \\ \hline
\text{43 (d.1)}
 &29619&58807&168410&0.252&0.3492&$\frac{589}{1681} = 0.3504$ \\ \hline
\text{47 (d.2)}
 &29220&58400&168365&0.250&0.3469&$\frac{767}{2207} = 0.3475$ \\ \hline
\text{53 (a.2)}
 &29278&58257&168353&0.252&0.3460&$\frac{9}{26} = 0.3462$ \\ \hline
\text{59 (d.2)}
 &29378&58422&168783&0.252&0.3461&$\frac{1199}{3479} = 0.3446$ \\ \hline
\text{61 (b.1)}
 &28027&55816&168197&0.251&0.3318&$\frac{1}{3} = 0.3333$ \\\hline
\text{67 (d.1)}
 &29242&57944&168239&0.253&0.3444&$\frac{1453}{4225} = 0.3439$ \\ \hline
\text{71 (c.2)}
 &28789&57661&168508&0.249&0.3422&$\frac{12}{35} = 0.3429$ \\ \hline
\text{73 (a.1)}
 &28975&57828&168614&0.251&0.3430&$\frac{12}{35} = 0.3429$ \\\hline
\text{79 (d.1)}
 &29127&57937&168690&0.252&0.3435&$\frac{2029}{5929} = 0.3422$ \\ \hline
\text{83 (d.2)}
 &29032&57871&168435&0.251&0.3436&$\frac{2351}{6887} = 0.3414$ \\ \hline
\text{89 (a.2)}
 &28909&57634&168737&0.251&0.3416&$\frac{15}{44} = 0.3409$ \\\hline
\text{97 (b.1)}
 &28014&55880&168457&0.251&0.3317&$\frac{1}{3} = 0.3333$ \\\hline
 \end{tabular}
\end{center}
\caption{Density of Amicable and  Type 1 primes with
 $p\le X= 10^8$ for the curve $y^2=x^3+k$, prime $k$.}
\label{tableCM9}
\end{table}

We also checked Conjecture~\ref{conjecture:N1kN1M1kM1} experimentally
for composite values of~$k$. The results are listed in
Table~\ref{tableCM10}, where the conjectural limiting ratio is
obtained by explicitly counting the size of the sets~$\Mcal_k$
and~$\Mcal_k^{[1]}$. The top eight $k$ entries in this table are
products of two primes covering the usual eight cases; the final four
entries include two values that are not square-free ($175=5\cdot7^2$
and $245=5\cdot7^2$) and two values that are products of three distinct
primes ($385=5\cdot7\cdot11$ and $455=5\cdot7\cdot13 $).
\par
In order to further test the validity of
Conjecture~\ref{conjecture:N1kN1M1kM1} , we recomputed the final entry
in the table with $X=10^9$ and obtained
\[
  \Ncal_{455}^{[1]}(10^9)/\Ncal_{455}(10^9) = 0.3380.
\]
This is in excellent agreement with the theoretical value
of $\frac{4699}{13915} = 0.3377$.

\begin{table}[ht]
\tiny
\begin{center}
\renewcommand{\arraystretch}{1.25}
\begin{tabular}{|r||*{10}{c|}} 
  \cline{6-7}
  \multicolumn{5}{c}{} & \multicolumn{2}{|c|}{Density of Type 1 primes} \\
  \multicolumn{5}{c}{} & \multicolumn{2}{|c|}{$\Ncal_k^{[1]}(X)/\Ncal_k(X)$} \\
  \hline
  \hfill\hfill$k$\hfill\hfill
   & $\Qcal_k(X)$ & $\Ncal_k^{[1]}(X)$ & $\Ncal_k(X)$ 
   & $\Qcal/\Ncal^{[1]}$ & experiment & conjecture \\
  \hline\hline
  \text{35 (d.2)}
    & 31423 & 63169 & 168666 & 0.248 & 0.3745 & $\frac{43}{115} = 0.3739$ 
      \\ \hline
  \text{55 (d.1)}
    & 29645 & 58718 & 168870 & 0.253 & 0.3477 & $\frac{949}{2737} = 0.3467$ 
      \\ \hline
  \text{77 (b.2)}
    & 28170 & 56251 & 168921 & 0.251 & 0.3330 & $\frac{1}{3} = 0.3333$ \\ \hline
  \text{85 (b.1)}
    & 28187 & 56142 & 168767 & 0.251 & 0.3327 & $\frac{1}{3} = 0.3333$ \\ \hline
  \text{323 (c.2)}
    & 28396 & 56609 & 168585 & 0.251 & 0.3358 & $\frac{43}{128} = 0.3359$ 
      \\ \hline
  \text{629 (a.2)}
    & 28210 & 56269 & 168042 & 0.251 & 0.3349 & $\frac{3267}{9766} = 0.3345$ 
      \\ \hline
  \text{703 (c.1)}
    & 28558 & 56754 & 168817 & 0.252 & 0.3362 & $\frac{1097}{3278} = 0.3347$ 
      \\ \hline
  \text{901 (a.1)}
    & 28341 & 56384 & 168411 & 0.252 & 0.3348 & $\frac{3738}{11189} = 0.3341$
       \\ \hline \hline
  \text{175 (d.1)}
    & 31543 & 63177 & 168840 & 0.250 & 0.3742 & $\frac{43}{115} = 0.3739$ 
        \\ \hline
  \text{245 (b.2)}
    & 29722 & 58848 & 175934 & 0.253 & 0.3345 & $\frac{1}{3} = 0.3333$ 
       \\ \hline
  \text{385 (b.1)}
    & 28070 & 56158 & 168393 & 0.250 & 0.3335 & $\frac{1}{3} = 0.3333$ 
       \\ \hline
  \text{455 (d.2)}
    & 28346 & 56627 & 168342 & 0.250 & 0.3364 & $\frac{4699}{13915} = 0.3377$
         \\ \hline
 \end{tabular}
\end{center}
\caption{Density of Amicable and  Type 1 primes with
 $p\le X= 10^8$ for the curve $y^2=x^3+k$, composite $k$.}
\label{tableCM10}
\end{table}

Finally, we consider Conjecture~\ref{conjecture:aliquotcount}, which
deals with non-CM curves. This conjecture is much harder to check
numerically because the function~$\sqrt{X}/(\log X)^2$ grows quite
slowly.  We performed an extended search for amicable pairs on the
elliptic curve
\begin{equation}
  \label{eqn:Cond43curve}
  E:y^2+y=x^3+x^2
\end{equation}
of conductor~$43$ that we studied in Example~\ref{example:cond3743}.
Appendix~\ref{appendix:pairsoncond43} lists all normalized amicable
pairs~$(p,q)$ with $p<10^{11}$.
Conjectures~\ref{conj:amicablepair}(a)
and~\ref{conjecture:aliquotcount} says that~$\Qcal_E(X)$, the number
of amicable pairs up to~$X$, should grow like a multiple of
$\sqrt{X}/(\log X)^2$.  Table~\ref{table:amicpair43} tests this
conjecture by computing the ratios
\[
  \frac{\Qcal_E(X) }{\sqrt{X}/2(\log X)^2}
  \qquad\text{and}\qquad
  \frac{\log \Qcal_E(X)}{\log X}
\]
for various values of~$X$. The computations for
Table~\ref{table:amicpair43}, i.e., the computation of
$\Qcal_E(10^{11})$, took approximately five days running in parallel
on a cluster of ten machines.  The third column of
Table~\ref{table:amicpair43} provides some small support for the
conjecture that~$\Qcal(X)$ grows like a multiple of $\sqrt{X}/(\log
X)^2$. On the other hand, although the fourth column of the table
suggests that~$\Qcal(X)$ grows like~$X^\d$ for some $\d>0$, it is far
from clear that~$\d$ is as large as~$\frac12$. We suspect the problem
is that we are only able to compute~$\Qcal(X)$ up to $X=10^{11}$, and
although~$10^{11}$ is a moderately large number in terms of
computation time, it is comparatively small compared to the
likely error terms in any putative asymptotic formula for~$\Qcal(X)$.

\begin{table}[ht]
\begin{center}
\small
\renewcommand{\arraystretch}{1.25}
\begin{tabular}{|c|c|c|c|c|c|c|} \hline
  $X$ & $\Qcal(X)$ & $\Qcal(X)\big/\frac{\sqrt{X}}{(\log X)^2}$
      & $\frac{\log \Qcal(X)}{\log X} $ \\ \hline\hline
  $10^{6}$ &   2  & 0.382 & 0.050 \\ \hline
  $10^{7}$ &   4  & 0.329 & 0.086 \\ \hline
  $10^{8}$ &   5  & 0.170 & 0.087 \\ \hline
  $10^{9}$ &   10 & 0.136 & 0.111 \\ \hline
  $10^{10}$ &  21 & 0.111 & 0.132 \\ \hline
  $2\cdot 10^{10}$ &  32 & 0.127 & 0.146 \\ \hline
  $4\cdot 10^{10}$ &  37 & 0.110 & 0.148 \\ \hline
  $6\cdot 10^{10}$ &  44 & 0.111 & 0.152 \\ \hline
  $8\cdot 10^{10}$ &  53 & 0.118 & 0.158 \\ \hline
  $ 10^{11}$ &  55 & 0.112 & 0.158 \\ \hline
\end{tabular}
\end{center}
\caption{Counting amicable pairs for $y^2+y=x^3+x^2$}
\label{table:amicpair43}
\end{table}

Finally, we searched for normalized aliquot triples~$(p,q,r)$ on the
curve~\eqref{eqn:Cond43curve}. We found no examples with $p<10^8$.
To compare this to Conjecture~\ref{conjecture:aliquotcount}, we consider
the counting functions 
\[
   \Pcal_E^{[\ell]}(X) 
    = \#\left\{ (p_1,\ldots,p_\ell) : 
    \begin{tabular}{@{}l@{}}
      $p_i$ distinct primes with $p_1\le X$ and\\
        $\#\Etilde_{p_i}(\FF_{p_i})=p_{i+1}$ for $1\le i\le \ell-1$\\
    \end{tabular}
 \right\}.
\]
Note that $\Pcal_E^{[\ell]}(X)$ is not counting aliquot cycles,
because we do not require that $\#\Etilde_{p_\ell}(\FF_{p_\ell})$
equal~$p_1$. A natural generalization of the Koblitz--Zywina
conjecture (Conjecture~\ref{conj:kobzy}) is that
\[
   \Pcal_E^{[\ell]}(X) \sim C_{E/\QQ}^{[\ell]} \frac{X}{(\log X)^{\ell}}.
\]
We computed $\Pcal_E^{[2]}(X)$ and $\Pcal_E^{[3]}(X)$
for the curve~\eqref{eqn:Cond43curve} and used it to estimate
the values of the constants~$C_{E/\QQ}^{[2]}$ and~$C_{E/\QQ}^{[3]}$.
The results are listed in Table~\ref{table:N2N3C2C3}.

\begin{table}[ht]
\small
\[
  \begin{array}{|c|c|c|c|c|c|} \hline
    X & \Pcal_E^{[2]}(X) & \Pcal_E^{[3]}(X) & C_{E/\QQ}^{[2]} & C_{E/\QQ}^{[3]}
      & \frac12\Pcal_E^{[3]}(X)X^{-1/2} \\
    \hline
    10^5 & 485 & 21 & 0.643 & 0.320 & 0.033  \\ \hline
    10^6 & 3099 & 116 & 0.592 & 0.306 & 0.058  \\ \hline
    10^7 & 22328 & 741 & 0.580 & 0.310 & 0.117  \\ \hline
    10^8 & 168611 & 4995 & 0.572 & 0.312 & 0.250  \\ \hline
  \end{array}
\]
\caption{An aliquot triple estimate for $y^2+y=x^3+x^2$}
\label{table:N2N3C2C3}
\end{table}

Our heuristic argument from Section~\ref{section:nonCM} suggests that
if~$(p,q,r)$ is a triple counted in~$\Pcal_E^{[3]}(X)$, then the
probability that~$\#\Etilde_r(\FF_r)$ equals~$p$ is~$O(r^{-1/2})$, and
more precisely, we expect it to be between~$\frac12r^{-1/2}$
and~$\frac14r^{-1/2}$. Thus the number of aliquot triples less
than~$X$ should roughly be somewhere between
\begin{equation}
  \label{eqn:12NE3}
  \frac12 \Pcal_E^{[3]}(X) X^{-1/2}
  \qquad\text{and}\qquad
  \frac14 \Pcal_E^{[3]}(X) X^{-1/2}.
\end{equation}
The last column of Table~\ref{table:N2N3C2C3} gives our heuristic
upper bound for the number of aliquot triples up to~$X$.  It is thus
not surprising that we found no examples up to~$10^8$.  Using the
estimate~$C_{E/\QQ}^{[3]}\approx0.312$, setting $\Pcal_E^{[3]}(X)
\approx C_{E/\QQ}^{[3]} \frac{X}{(\log X)^{3}}$, and taking the upper
bound in~\eqref{eqn:12NE3}, it would require at least
$X\approx5\cdot10^9$ in order to expect find an aliquot triple.  And
even for~$X=10^{11}$ we wouldn't expect more than two or three.

\section{Motivation and generalizations}
\label{section:tangentialremarks}

\begin{remark}
\label{remark:motivation}
Elliptic amicable pairs and aliquot cycles appeared in a natural
fashion when the authors were generalizing to elliptic divisibility
sequences Smyth's results~\cite{Smyth} on index divisibility of
Lucas sequences. Let~$(W_n)_{n\ge1}$ be
a normalized nonsingular nonperiodic elliptic divisibility sequence
associated to an elliptic curve~$E/\QQ$, and consider the set
\[
  \Scal=\{n\ge1 : n\mid W_n\}. 
\]
An element of~$\Scal$ is called \emph{basic} if $n/p\notin\Scal$ for
all primes$p$ dividing~$n$.  
It turns out that basic elements can be created using
aliquot cycles, a phenomenon that does not occur in Smyth's work.
More precisely, given any aliquot cycle~$(p_1,\ldots,p_\ell)$
for~$E/\QQ$ of length $\ell\ge2$, the product~$p_1p_2\cdots p_\ell$ is
a basic element for~$(W_n)$, and more generally, any product of such
products is basic.  See~\cite{SilvStang} for details.
\end{remark}

\begin{remark}
\label{remark:genECaliq}
As we have defined them, aliquot cycles for elliptic curves differ in
a significant way from classical aliquot cycles associated to the sum
of divisors function. In the classical case, every integer~$n$ leads
to a possibly non-repeating aliquot sequence
$\bigl(n,\stilde(n),\stilde^2(n),\stilde^3(n),\ldots\bigr)$, and it is
an aliquot cycle if some iterate~$\stilde^k(n)$ eventually returns
to~$n$. (A major open problem is whether there are starting values for
which the sequence is unbounded.) But for elliptic curves, if we
arrive at a prime~$p$ such that~$\#\Etilde_p(\FF_p)$ is not prime,
then the sequence cannot be continued.  We propose here two
alternative definitions of elliptic aliquot sequences that more
closely resemble the classical definition.  We leave the investigation
of these generalized sequences to a future paper.
\end{remark}

\begin{definition}
Let~$E/\QQ$ be an elliptic curve, let $L(E/\QQ,s)=\sum_{n\ge1}a_n/n^s$
be the $L$-series of~$E$, and define a function
\[
  F_E : \NN \longrightarrow\NN,\qquad
  F_E(n) = n+1-a_n.
\]
A \emph{Type L aliquot sequence} for~$E/\QQ$ is a sequence
obtained by starting at some $n\in\NN$ and repeatedly applying the
map~$F_E$. A \emph{Type L aliquot cycle} is a Type~$L$ aliquot sequence that
repeats at its starting value. 
\end{definition}

\begin{definition}
Let~$E/\QQ$ be an elliptic curve, let~$\Ecal^0/\ZZ$ be the open subset
of the N\'eron model for~$E/\QQ$ consisting of the connected
components of each fiber, and define a function
\[
  G_E : \NN \longrightarrow\NN,\qquad
  G_E(n) = \#\Ecal^0(\ZZ/n\ZZ).
\]
A \emph{Type N aliquot sequence} for~$E/\QQ$ is a sequence obtained by
starting at some $n\in\NN$ and repeatedly applying the 
map~$G_E$. A \emph{Type N aliquot cycle} is a Type~$N$ aliquot
sequence that repeats at its starting value. 
\end{definition}

\begin{remark}
\label{remark:numberfield}
There is a natural way to generalize the notion of elliptic amicable
pairs and aliquot cycles to elliptic curves defined over number
fields. Thus let~$F/\QQ$ be a number field and~$E/F$ an elliptic
curve. We will say that a sequence of distinct degree one prime
ideals~$\gp_1,\gp_2,\ldots,\gp_\ell$ is an \emph{aliquot cycle} of
length~$\ell$ for~$E/F$ if~$E$ has good reduction at every~$\gp_i$ and
\begin{multline*}
  \#\Etilde_{\gp_1}(\FF_{\gp_1}) = \Norm_{K/\QQ}(\gp_2),\quad
  \#\Etilde_{\gp_2}(\FF_{\gp_2}) = \Norm_{K/\QQ}(\gp_3),\quad \ldots \\
  \#\Etilde_{\gp_{\ell-1}}(\FF_{\gp_{\ell-1}}) = \Norm_{K/\QQ}(\gp_\ell),\quad
  \#\Etilde_{\gp_\ell}(\FF_{\gp_\ell}) = \Norm_{K/\QQ}(\gp_1).
\end{multline*}
Many of the methods and results in this paper carry over in a
straightforward manner to the number field case.  For example, the
following analogue of Theorem~\ref{thm:amicableCMconj}
holds.

\begin{theorem}
Let $F/\QQ$ be a number field, and let $E/F$ be an elliptic curve with
complex multiplication by an order in the quadratic imaginary
field~$K$. Suppose that $\gp$ and $\gq$ are degree one primes of~$F$
at which~$E$ has good reduction,  that $\Norm_{F/\QQ}\gp\ge5$,
and that
\[
  \#\Etilde_{\gp}(\FF_{\gp}) = \Norm_{F/\QQ}\gq.
\]
Assume further that $j(E)\ne0$. Then
\[
  \#\Etilde_{\gq}(\FF_{\gq}) = \Norm_{F/\QQ}\gp
  \qquad\text{or}\qquad
  \#\Etilde_{\gq}(\FF_{\gq}) = 2\Norm_{F/\QQ}\gq+2-\Norm_{F/\QQ}\gp.
\]
\end{theorem}

It would be interesting to see to what extent the other results in
this paper are valid over number fields, including especially the
analysis of amicable pairs on curves with $j(E)=0$.

\end{remark}


\begin{acknowledgement}
The authors would like to thank Franz Lemmermeyer, Jonathan Wise, and
Soroosh Yazdani for their assistance.  The research in this note was
performed while the first author was a long-term visiting researcher
at Microsoft Research New England and included a short visit by the
second author. Both authors thank MSR for its hospitality during
their visits.
\end{acknowledgement}




\appendix

\section{Curves with $j=0$ have no aliquot triples}
\label{appendix:j0aliqtriples}
In this section we use Corollary~\ref{corollary:j0explicit} and a
detailed case-by-case analysis to show that an elliptic curve with
$j=0$ has no normalized aliquot triples~$(p,q,r)$ with $p>7$. The
details are sufficiently intricate that it seems likely a different
argument would be needed to prove that there are no aliqout cycles of
length greater than three.

\begin{proposition}
Let $E/\QQ$ be an elliptic curve with $j(E)=0$. Then~$E$ has
no normalized aliquot triples $(p,q,r)$ with $p > 7$.
\end{proposition}
\begin{proof}
We use Corollary~\ref{corollary:j0explicit}, which says that if~$p$
and $q=\#\Etilde_p(\FF_p)$ are prime, then $r=\#\Etilde_q(\FF_q)$
takes one of six possible values. One of these six possible values
is~$p$, which is not allowed since we are assuming that~$p,q,r$ are
distinct. Hence~$r$ has one of the following forms,
\begin{align*}
  r &= 2q+2-p, &&\text{(Case 1)} \\
  r &= \frac{\pm(q+1-p)\pm3A_{p,q}}{2}, &&\text{(Case 2)} 
\end{align*}
where~$A_{x,y}$ satisfies
\[
  A_{x,y}^2 = \frac{4xy-(x+y-1)^2}{3}.
\]
(Of course, Case~2 is really four cases, depending on the choice of
signs.)
\par
For the moment letting $s=\#\Etilde_r(\FF_r)$, we can apply the same
reasoning to $(q,r,s)$ to deduce that
\begin{align*}
  s &= 2r+2-q, &&\text{(Case A)} \\
  s &= \frac{\pm(r+1-q)\pm3A_{q,r}}{2}. &&\text{(Case B)} 
\end{align*}
To ease notation, we let
\[
  F(x,y) = \frac{\pm(y+1-x)\pm3A_{x,y}}{2}.
\]
Then the two cases for~$r$ followed by the two cases for~$s$ give
four possibilities for~$s$ in terms of~$p$ and~$q$:
\begin{align*}
  s &= 3q+6-2p, &&\text{(Case 1A)} \\
  s &= 2F(p,q)+2-q, &&\text{(Case 2A)} \\
  s &= F(q,2q+2-p), &&\text{(Case 1B)} \\
  s &= F\bigl(q,F(p,q)\bigr). &&\text{(Case 2B)}
\end{align*}
(Of course, each case is really several cases depending on the choice of
signs for each occurrence of~$F$.)
\par
The assumption that~$(p,q,r)$ is an aliquot triple is equivalent to saying
that~$s=p$. Suppose first we are in Case~1A. Then $s=p$ is equivalent
to
\[
  3q+6-2p = p,
\]
so $p=q+2$. This contradicts our assumption that the triple is normalized,
i.e., that~$p$ is the smallest element of the triple. Hence Case~1A is 
not possible.
\par
Next we consider Case~2A. Then the assumption $s=p$ implies
that $2F(p,q)=p+q-2$. Using the definition of~$F$, this can be written as
\[
  \pm(q+1-p)\pm3A_{p,q} = p+q-2,
\]
which (using the definition of~$A$) implies that
\begin{equation}
  \label{eqn:2Apm}
  \bigl((p+q-2)\pm(q+1-p)\bigr)^2 = 9A_{p,q}^2 = 3\bigl(4pq-(p+q-1)^2\bigr).
\end{equation}
This gives two subcases, which we denote by~2A$^+$ and~2A$^-$ according
to the choice of sign. A little bit of algebra yields
\begin{align*}
 28p^2 -24pq +12q^2 - 72p  - 24q + 48&=0, &&\text{(Case 2A$^+$)} \\
 12p^2 - 24pq + 28q^2 - 24p - 40q + 16&=0. &&\text{(Case 2A$^-$)}
\end{align*}
Both of the functions on the left-hand sides
have leading quadratic forms that are  positive definite, so there are
only finitely many integral solutions~$(p,q)$. A more careful analysis
shows that the first is positive for $p>5$ and the second is positive
for $p>7$.
\par
Next comes Case~1B, where the assumption that~$s=p$ leads to
the formula
\[
  F(q,2q+2-p)=p.
\]
Writing this out in terms of $A_{q,2q+2-p}$, moving all the other terms
to the other side, squaring, and simplifying, we again get two cases
depending a choice of sign. Thus
\begin{align*}
 12p^2-12pq+4q^2-24p+12&=0, &&\text{(Case 1B$^+$)} \\
 4p^2-4pq+4q^2+12&=0. &&\text{(Case 1B$^-$)} 
\end{align*}
The quadratic function for Case 1B$^+$ is positive for $p>7$
and the  quadratic function for Case 1B$^-$ is positive for $p>0$.
\par
Finally we turn to Case~2B, which is somewhat more complicated because
it is given by the formula
\[
  F\bigl(q,F(p,q)\bigr) = p,
\]
which involves two iterations of the function~$F$. The signs on
the~$A_{x,y}$ terms are irrelevant since we square them, but the other
signs in the definition of~$F$ do affect the eventual equation.
After a bunch of algebra, we find that the~$p$ and~$q$ values for an
amicable triple coming from Case~2B must satisfy one of the following
equations.
\begin{align*}
   4p^4 + 2p^3q + 3p^2q^2 - pq^3 + q^4 - 6p^3 - 15p^2q  \\
    {}  - 15pq^2 + 3p^2 +  3pq + 3q^2 &=0
  &&\text{(Case 2B$^{++}$)}  \\
  9p^2q^2 - 9pq^3 + 9q^4  + 9p^2q - 27pq^2  + 3p^2 - 21pq \\
    {} - 3q^2    - 6p + 6q + 4 &=0
  &&\text{(Case 2B$^{+-}$)} \\
 3p^2q^2 - 3pq^3 + q^4  + 9p^2q - 9pq^2  + 9p^2 - 9pq + 3q^2 &=0
  &&\text{(Case 2B$^{-+}$)} \\
 4p^4 - 18p^3q + 33p^2q^2 - 27pq^3 + 9q^4 - 10p^3  + 33p^2q   \\
  {} - 21pq^2 + 21p^2 - 21pq - 3q^2 - 10p + 6q + 4 &= 0
  &&\text{(Case 2B$^{--}$)}
\end{align*}
All of these quartic functions are positive if $0<p<q$ with~$p$
sufficiently large. More precisely, it suffices to take $p>3$ for
Cases 2B$^{++}$ and 2B$^{+-}$, $p>4$ for Case 2B$^{-+}$ and $p>2$ for
Case 2B$^{--}$.
\par
This completes the proof that~$E$ has no aliquot triples.
\end{proof}

\section{Proof of independence of the maps 
in Proposition~\ref{proposition:ECDcurves}}
\label{appendix:indepofmaps}
We give a proof for $\operatorname{char}(\FF)=0$.  Independence of
maps is geometric, so it suffices to prove independence for~$\g=\d=1$.
Let~$\psi_1,\ldots,\psi_4$ be the four maps in
Proposition~\ref{proposition:ECDcurves}(b), so
\begin{align*}
  \psi_1:C_6^{(1,1)} &\longrightarrow E^{(16)},
  &\psi_2:C_6^{(1,1)} &\longrightarrow E^{(4)},\\
  \psi_3:C_6^{(1,1)} &\longrightarrow E^{(1)},
  &\psi_4:C_6^{(1,1)} &\longrightarrow E^{(-1)}.
\end{align*}
We compose these maps with untwisting maps~$E^{(\k)}\to E^{(1)}$, so we get
four maps
\[
  \begin{CD}
    \f_1 : C_6^{(1,1)} @>{\psi_1}>> E^{(16)}
       @>{(x,y)\to
   \smash[t]{\left(\frac{1}{2\sqrt[3]{2}} x,\frac14y\right)}}>> E^{(1)}, \\
    \f_2 : C_6^{(1,1)} @>{\psi_2}>> E^{(4)}
       @>{(x,y)\to\left(\frac{1}{\sqrt[3]{4}} x,\frac12y\right)}>> E^{(1)}, \\
    \f_3 : C_6^{(1,1)} @>{\psi_3}>> E^{(1)}
       @>{(x,y)\to(x,y)}>> E^{(1)}, \\
    \f_4 : C_6^{(1,1)} @>{\psi_4}>> E^{(-1)}
       @>{(x,y)\to(-x,iy)}>> E^{(1)}.
  \end{CD}
\]
The maps~$\psi_1,\ldots,\psi_4$ are defined over~$\QQ$, but the
maps~$\f_1,\ldots,\f_4$ are only defined over~$\Qbar$, not~$\QQ$. We
consider the action of $\Gal(\Qbar/\QQ)$ on these maps. To do this, we
choose elements $\s,\t\in\Gal(\Qbar/\QQ)$ satisfying
\begin{align*}
  \s(\sqrt[3]{2})&=\rho\sqrt[3]{2},
  &\s(i)&=i,\\
  \t(\sqrt[3]{2})&=\sqrt[3]{2},
  &\t(i)&=-i.
\end{align*}
Here $\rho=\frac12(-1+\sqrt{-3})$ is a fixed primitive cube root of
unity.  We also note that~$\bfmu_3$ acts on~$E^{(1)}$ via $[\r](x,y) =
(\r x,y)$. Looking at the explicit formulas for~$\f_1,\ldots,\f_4$, we
find that
\begin{align*}
  \f_1^\s &= [\r^2]\circ\f_1, 
  &\f_2^\s &= [\r]\circ\f_2, 
  &\f_3^\s &= \f_3, 
  &\f_4^\s &= \f_4, \\*
  \f_1^\t &= \f_1, 
  & \f_2^\t &= \f_2, 
  & \f_3^\t &= \f_3, 
  & \f_4^\t &= [-1]\circ\f_4.
\end{align*}
Now suppose that we have a relation
\begin{equation}
  \label{eqn:n1f1n4f40}
  [n_1]\circ\f_1 + [n_2]\circ\f_2 + [n_3]\circ\f_3 + [n_4]\circ\f_4 = 0.
\end{equation}
Applying the transformation~$\t$ to~\eqref{eqn:n1f1n4f40} has the
effect of replacing~$\f_4$ by~$[-1]\circ\f_4$, so subtracting the two
equations yields $[2n_4]\circ\f_4=0$. Since the map
$\f_4:C_6^{(1,1)}\to E^{(1)}$ is a finite map, it follows
that~$n_4=0$.
\par 
Applying~$\s$ and~$\s^2$ to~\eqref{eqn:n1f1n4f40}, we end
up with three equations
\begin{align}
  \label{eqn:n1f1n4f40a}
  [n_1]\circ\f_1 + [n_2]\circ\f_2 + [n_3]\circ\f_3  &= 0, \\
  \label{eqn:n1f1n4f40b}
  [n_1]\circ[\r^2]\circ\f_1 + [n_2]\circ[\r]\circ\f_2 + [n_3]\circ\f_3  &= 0, \\
  \label{eqn:n1f1n4f40c}
  [n_1]\circ[\r]\circ\f_1 + [n_2]\circ[\r^2]\circ\f_2 + [n_3]\circ\f_3  &= 0.
\end{align}
Adding~\eqref{eqn:n1f1n4f40a},~\eqref{eqn:n1f1n4f40b},
and\eqref{eqn:n1f1n4f40c} and using $1+\r+\r^2=0$
gives~$[3n_3]\circ\f_3=0$, which implies that $n_3=0$.  Similarly,
adding~\eqref{eqn:n1f1n4f40a} to~$[\r]$ times~\eqref{eqn:n1f1n4f40b}
to~$[\r^2]$ times~\eqref{eqn:n1f1n4f40c} gives~$[3n_1]\circ\f_1=0$,
so~$n_1=0$. Finally, since $n_1=n_3=0$, the
equation~\eqref{eqn:n1f1n4f40a} gives~$n_2=0$. This completes the
proof that~$\f_1,\ldots,\f_4$ are independent.

\section{Evaluating the formula for $M_k^{[1]}$ when $k\equiv1\pmod3$}
\label{appendix:pariprogram}
In this appendix we give the PARI~\cite{PARI} script that we used to
compute $M_k^{[1]}(S,1)$ via the formulas~\eqref{eqn:splitct1},
\eqref{eqn:splitct2}, \eqref{eqn:splitct3}, and~\eqref{eqn:splitct4}.
In these formulas we treat~$k$,~$\pi$, and~$\e$ as indeterminates and
formally set $\bar\pi=k/\pi$ and $\bar\e=1/\e$.  The value of
$M_k^{[1]}(S,1)$ turns out to be independent of $k\bmod4$ and is
always a quadratic polynomial in~$k$.  The output from the routine
\texttt{TestFormulaForSplitPrimes} is given in
Table~\ref{table:parioutput}, which was calculated using the following
PARI program.

\begin{table}[ht]
{\small
\begin{align*}
  k\equiv 1 \pmod{4} \\
  \#M_k^{[1]}(\o^0,1) &= (1/18)(k^2 + 4k + 13) \\
  \#M_k^{[1]}(\o^1,1) &= (1/18)(k^2 - 2k + 1) \\
  \#M_k^{[1]}(\o^2,1) &= (1/18)(k^2 - 2k + 1) \\
  \#M_k^{[1]}(\o^3,1) &= (1/18)(k^2 - 8k + 7) \\
  \#M_k^{[1]}(\o^4,1) &= (1/18)(k^2 - 2k + 1) \\
  \#M_k^{[1]}(\o^5,1) &= (1/18)(k^2 - 2k + 1) \\
  \\
  k\equiv 3 \pmod{4} \\
  \#M_k^{[1]}(\o^0,1) &= (1/18)(k^2 + 4k + 13) \\
  \#M_k^{[1]}(\o^1,1) &= (1/18)(k^2 - 2k + 1) \\
  \#M_k^{[1]}(\o^2,1) &= (1/18)(k^2 - 2k + 1) \\
  \#M_k^{[1]}(\o^3,1) &= (1/18)(k^2 - 8k + 7) \\
  \#M_k^{[1]}(\o^4,1) &= (1/18)(k^2 - 2k + 1) \\
  \#M_k^{[1]}(\o^5,1) &= (1/18)(k^2 - 2k + 1)
\end{align*}
\begin{align*}
  \#M_k^{[1]}\bigl(\{\o^1,\o^5\},1\bigr)
     &= (1/9)(k^2 - 2k + 1) \\
  \#M_k^{[1]}\bigl(\{\o^1,\o^3,\o^5\},1\bigr)
     &= (1/6)(k^2 - 4k + 3) \\
  \#M_k^{[1]}\bigl(\{\o^1,\o^2,\o^4,\o^5\},1\bigr)
     &= (2/9)(k^2 - 2k + 1) \\
  \#M_k^{[1]}\bigl(\{\o^0,\o^1,\o^2,\o^3,\o^4,\o^5\},1\bigr)
     &= (1/3)(k^2 - 2k + 4)
\end{align*}
\par}
\caption{Results of computing of $M_k^{[1]}(S,1)$ using PARI}
\label{table:parioutput}
\end{table}

{\tiny
\begin{verbatim}
/* The function TestFormulaForSplitPrimes computes M_k^{[1]} 
   for k = 1 (mod 3), using the formula as a sum of products of #C - e terms.
   The following global variables must be left as indeterminates:
   k, pi, e2
   Here e2 represents the cubic residue of 2 modulo pi.
   The conjugates of these are given by
   pibar = k/pi  and   e2bar = 1/e2.
   Further, w is assigned the value quadgen(-3), and wbar = conj(w) = 1/w
*/

{
TestFormulaForSplitPrimes() =
  local(m,ISets,IndexSet);
  print("\\begin{align*}");
  forstep(kmod4 = 1, 3, 2,
    print("  k\\equiv ",kmod4," \\pmod{4} \\\\");
    for (i = 0, 5,
      m = MMSum([i],kmod4);
      print1("  \\#M_k^{[1]}(\\o^",i,",1) &= ");
      print1("(",content(content(m)),")(",m/content(content(m)),")");
      if (i < 5 || kmod4 == 1, print(" \\\\"), print);
    );
    if (kmod4 == 1, print("  \\\\"));
  );
  print("\\end{align*}");
  print("\\begin{align*}");
  ISets = [[1,5], [1,3,5], [1,2,4,5], [0,1,2,3,4,5]];
  for (j = 1, #ISets,
    IndexSet = ISets[j];
    m = MMSum(IndexSet,1);
    print1("  \\#M_k^{[1]}\\bigl(\\{");
    for (i = 1, #IndexSet,
      print1("\\o^",IndexSet[i]);
      if (i < #IndexSet, print1(","));
    );
    print1("\\},1\\bigr)\n     &= (",content(content(m)),")");
    print1("(",m/content(content(m)),")");
    if (j < #ISets, print(" \\\\"), print);
  );
  print("\\end{align*}");
}

w = quadgen(-3);
wbar = conj(w);
pibar = k/pi;
e2bar = 1/e2;

{
CC(u,v,kmod4) =
  if (kmod4 == 0, error("k mod 4 must be 1 or 3"));
  k + 1
     + w^(2*v)*pibar + wbar^(2*v)*pi
     + e2^2*w^(3*u+4*v)*pibar + e2bar^2*wbar^(3*u+4*v)*pi
     + e2*w^(5*u+2*v)*pibar + e2bar*wbar^(5*u+2*v)*pi
     + (-1)^((kmod4-1)/2)*e2*w^(u+2*v)*pibar 
     + (-1)^((kmod4-1)/2)*e2bar*wbar^(u+2*v)*pi;
}

{
ee(u,v) =
  if (u % 6 == 0, 6, 0)
    + if((u-v) % 3 == 0, 3, 0)  + if((2*u-v) % 3 == 0, 3, 0);
}

MM(u,v,kmod4) = (1/18) * (CC(u,v,kmod4) - ee(u,v));

{
MMSum(IndexSet,kmod4) =
  local(i,s);
  s = 0;
  for (j = 1, #IndexSet,
    i = IndexSet[j];
    for (u = 0, 5,
    for (v = 0, 2,
      s = s + MM(u,v,kmod4)*MM(u-i,v,kmod4);
    );
    );
  );
  return(s);
}
\end{verbatim}
}

\section{Amicable pairs for $y^2+y=x^3+x^2$ up to $10^{11}$}
\label{appendix:pairsoncond43}
We used PARI-GP~\cite{PARI} to compute all normalized amicable pairs
$(p,q)$ on the curve $y^2+y=x^3+x^2$ with $p<10^{11}$.
The list is given in Table~\ref{table:amicpair43list}.

\begin{table}[ht]
\begin{center}
\small
\begin{tabular}{|r|r|} \hline
 (853,883) & (77761,77999) \\
 (1147339,1148359) & (1447429,1447561) \\
 (82459561,82471789) & (109165543,109180121) \\
 (253185307,253194619) & (320064601,320079131) \\
 (794563993,794571803) & (797046407,797057473) \\
 (2185447367,2185504261) & (2382994403,2383029443) \\
 (4101180511,4101190039) & (4686466159,4686510971) \\
 (5293671709,5293749623) & (6677602471,6677694539) \\
 (7074693823,7074704971) & (7806306133,7806380963) \\
 (9395537549,9395559011) & (9771430993,9771433303) \\
 (9849225103,9849306373) & (10574564857,10574619851) \\
 (12657210407,12657303353) & (13003880317,13003900901) \\
 (13789895011,13790023199) & (14436076927,14436180091) \\
 (14976551207,14976590371) & (15597047659,15597075937) \\
 (15679549877,15679688491) & (16322301811,16322366867) \\
 (17725049203,17725142719) & (17841395323,17841406601) \\
 (31615097957,31615194739) & (33266376239,33266419807) \\
 (33963999907,33964128017) & (34525477799,34525684639) \\
 (39287748091,39287808559) & (40136806357,40137038941) \\
 (46438194193,46438453213) & (51838270219,51838493561) \\
 (51881025571,51881167549) & (52011956957,52012184953) \\
 (55823622193,55823919169) & (57920520199,57920640709) \\
 (62765305697,62765625749) & (62995853671,62996152237) \\
 (66252308051,66252349753) & (67177409329,67177631771) \\
 (69449506103,69449741239) & (75002612911,75002660263) \\
 (77264683829,77264993327) & (77635421531,77635670141) \\
 (79067605783,79067881429) & (81263083703,81263204563) \\
 (94248260597,94248586591) &  \\
  \hline
\end{tabular}
\end{center}
\caption{Amicable pairs for $y^2+y=x^3+x^2$ up to $10^{11}$}
\label{table:amicpair43list}
\end{table}

\end{document}